\documentclass[pdflatex,10pt]{article}
\usepackage[utf8]{inputenc}
\usepackage{mathtools}
\usepackage{amsmath,color}
\usepackage{amssymb,amsthm}
\usepackage{subfig}
\usepackage[permil]{overpic}
\usepackage[multiple]{footmisc}
\usepackage{xr}
\usepackage[left=1.6cm,right=1.6cm,top=2.50cm,bottom=2.50cm]{geometry}
\usepackage{graphicx}
\usepackage[font=small,labelfont=bf,
   justification=justified,
   format=plain,labelsep=space]{caption}
\usepackage{float}
\usepackage[greek,english]{babel}
\usepackage{teubner}
\usepackage[font=small,labelfont=bf,justification=centering]{caption}
\usepackage[T1]{fontenc}
\usepackage{lastpage}
\usepackage{enumerate}
\usepackage{enumitem}
\usepackage{lmodern}
\usepackage{array}
\usepackage{bm}
\usepackage{multicol}
\usepackage{multirow}
\usepackage{dsfont}
\usepackage{tensor}
\usepackage{fancyhdr}
\usepackage{listings}
\usepackage{dsfont}
\usepackage{siunitx}
\usepackage{titling}
\usepackage{lipsum}
\usepackage{tabularx}
\usepackage{verbatim}
\usepackage{authblk}
\usepackage{csquotes}
\usepackage{chemfig} %for chemistry figures
\usepackage{bbm} %for indicator function
\usepackage{hyperref} % good for references
\usepackage{ctable} % for \specialrule command for tables
\usepackage{multirow} % Merge rows in tables 
\usepackage{titlesec} % Modify font sections, subsections, etc.
\usepackage{marvosym}
\usepackage{relsize,exscale} % for big integrals
\usepackage{setspace}
\usepackage{filecontents}

\allowdisplaybreaks

\graphicspath{{Fig/}}

\addto\captionsenglish{}

\titleformat{\subsection}{\normalfont\large\raggedright\it}{\thesubsection}{1em}{}

\makeatletter
\newcommand{\thickhline}{%
    \noalign {\ifnum 0=`}\fi \hrule height 1pt
    \futurelet \reserved@a \@xhline
}
\newcolumntype{"}{@{\hskip\tabcolsep\vrule width 1pt\hskip\tabcolsep}}
\makeatother

\usepackage{adjustbox}

\makeatletter
\newsavebox{\@brx}
\newcommand{\llangle}[1][]{\savebox{\@brx}{\(\m@th{#1\langle}\)}%
  \mathopen{\copy\@brx\kern-0.5\wd\@brx\usebox{\@brx}}}
\newcommand{\rrangle}[1][]{\savebox{\@brx}{\(\m@th{#1\rangle}\)}%
  \mathclose{\copy\@brx\kern-0.5\wd\@brx\usebox{\@brx}}}
\makeatother

\newtheorem{deff}{Definition}[section]

\newtheorem{thm}[deff]{Theorem}

\newtheorem{cor}[deff]{Corollary}
\theoremstyle{definition}

\numberwithin{equation}{section}

\usepackage{graphicx} % Required for inserting images

\title{Critical Slowing Down in Bifurcating Stochastic Partial \\Differential Equations with Red Noise}
\author[1,$\dag$]{P. Bernuzzi}
\author[1,2]{C. Kuehn}
\author[3,4]{A. Morr}
\affil[1]{\footnotesize{Technical University of Munich, School of Computation Information and Technology, Department of
Mathematics, Boltzmannstraße 3, 85748 Garching, Germany}}
\affil[2]{\footnotesize{Technical University of Munich, Munich Data Science Institute, Walther-von-Dyck-Straße 10, 85748 Garching, Germany}}
\affil[3]{\footnotesize{Technical University of Munich, School of Engineering and Design, Lise-Meitner Str. 9, 85521 Ottobrunn, Germany}}
\affil[4]{\footnotesize{Potsdam Institute for Climate Impact Research, Research Domain IV: Complexity Science, Telegrafenberg A 31, 14473 Potsdam, Germany}}
\affil[$\dag$]{Author to whom any correspondence should be addressed.\smallskip 

Email addresses: paolo.bernuzzi@ma.tum.de (P. Bernuzzi), ckuehn@ma.tum.de (C. Kuehn), andreas.morr@tum.de (A. Morr).}
\date{\today}

\begin{document}

\maketitle

% REQUIRED

\begin{abstract}
The phenomenon of critical slowing down (CSD) has played a key role in the search for reliable precursors of catastrophic regime shifts. This is caused by its presence in a generic class of bifurcating dynamical systems. Simple time-series statistics such as variance or autocorrelation can be taken as proxies for the phenomenon, making their increase a useful early warning signal (EWS) for catastrophic regime shifts. However, the modelling basis justifying the use of these EWSs is usually a finite-dimensional stochastic ordinary differential equation, where a mathematical proof for the aptness is possible. Only recently has the phenomenon of CSD been proven to exist in infinite-dimensional stochastic partial differential equations (SPDEs), which are more appropriate to model real-world spatial systems. In this context, we provide an essential extension of the results for SPDEs under a specific noise forcing, often referred to as red noise. This type of time-correlated noise is omnipresent in many physical systems, such as climate and ecology. We approach the question with a mathematical proof and a numerical analysis for the linearised problem. We find that also under red noise forcing, the aptness of EWSs persists, supporting their employment in a wide range of applications. However, we also find that false or muted warnings are possible if the noise correlations are non-stationary. We thereby extend a previously known complication with respect to red noise and EWSs from finite-dimensional dynamics to the more complex and realistic setting of SPDEs.
\end{abstract}

\vspace{2em}
\small{\textbf{Keywords:} SPDEs, Early-warning signs, Critical Slowing Down, Red noise, Bifurcations.}

\small{\textbf{MSC codes:} 
60H15   % Stochastic partial differential equations

\small{\textbf{Funding:} This work was supported by the European Union’s Horizon 2020 research and innovation programme under Grant Agreement 956170.
This is ClimTip contribution \#. The ClimTip project has received funding from the European Union's Horizon Europe research and innovation programme under grant agreement No.~101137601. Views and opinions expressed are however those of the author(s) only and do not necessarily reflect those of the European Union or the European Climate, Infrastructure, and Environment Executive Agency (CINEA). Neither the European Union nor the granting authority can be held responsible for them.}}

% Header
\pagestyle{fancy}
\fancyhead{}
\renewcommand{\headrulewidth}{0pt}
\fancyhead[C]{\textit{Early Warning Signs for SPDEs with Red Noise}}

\section{Introduction}
Identifying and assessing high-impact events, such as climate tipping points or ecological catastrophes, has gained tremendous importance recently~\cite{Lenton2008, ArmstrongMcKay2022TippingPoints}. The non-linear, potentially abrupt response of complex systems to anthropogenic changes poses a grave challenge for conventional modelling techniques \cite{Valdes2011BuiltStability}. This has motivated the use of conceptual models, wherein the complex, high-dimensional dynamics are reduced to a few variables and their characterising feedback mechanisms. Often, a noise model is added to replace some of the omitted complexity \cite{Melbourne2011ChaosToNoise, Zwanzig2001Formalism}. In such conceptual models, the nature of abrupt regime shifts can be formalized as dynamical bifurcations, that is, a vanishing of a previously stable equilibrium state \cite{Alkhayuon2019AMOCRateTipping, Boers2017DeforAmaz}. 

Using the effect of noise, there is a particular phenomenon to be observed before these bifurcations. This phenomenon is often called critical slowing down (CSD) \cite{Kuehn2011CSD} since it describes the weakening response of the destabilizing system to noise disturbances. CSD has been taken as the basis for many statistical early warning signals (EWSs) in time series data. An increase in variance is often interpreted as an ongoing destabilization and an impending system collapse \cite{Boers2018DOEWS, Brett2020ROCEWS}. However, the exact behaviour of the system under noise close to the bifurcation depends on the type of noise \cite{Kuehn2022ColourBlind, Morr2022RedNoise, Benson2024alphaStableCSD}. One specific noise model, referred to as red noise, can represent time 
correlation in the dynamics and is, therefore, suitable for a large range of applications \cite{Hasselmann1976Theory1, Vaughan2005RedNoiseAstro, Santos2019RedNoiseElectronics}. In the context of CSD, it can be shown that an increase in variance precedes a bifurcation, but only under the assumption of stationary noise \cite{Morr2022RedNoise}. If the correlation time of the noise is itself subject to change, the EWS may be false or muted \cite{Clarke2023ROSA}.

An additional challenge to the validity of EWSs is the simplicity of the employed low-dimensional conceptual model \cite{Ben-Yami2024TippingTime, Lohmann2024AMOCMultistability}. In particular, it is often impossible to rigorously derive a low-dimensional model from the real-world system, which will generally be a spatial and infinite-dimensional dynamical system. A system of partial differential equations (PDEs) can more accurately describe these dynamics \cite{Dijkstra1997AMOCPDE}. However, even then, a stochastic noise term may be needed to replace fast time scale dynamics beyond the model's resolution\cite{Baars2017SPDEBoundary, Pal2022TippingSpatialRedNoise}. Such stochastic partial differential equations (SPDEs) have only recently been investigated in the context of EWSs for bifurcations \cite{bernuzzi2023bifurcations,Bernuzzi2024EWSSPDEContinuousSpectrum, Bernuzzi2024EWSSPDEBoundary, Gowda2015EWSSPDE, Romano2019ScalLawsAndEWSs}. These works supply a mathematical basis for the use of EWSs in a considerably larger range of dynamical systems. The investigated noise model is the canonical white noise model applied to the domain and boundary dynamics, respectively.

In this work, we advance the previous findings on the applicability of variance and autocorrelation-based EWSs. As discussed, this has already been achieved for low-dimensional models with red noise and for SPDEs with white noise. Here, we investigate SPDEs under the influence of red noise. This step is essential in assuring the aptness of the CSD and EWS framework in real-world systems. This is because real-world systems will generally be infinite-dimensional, and their chaotic, fast-time scale dynamics will typically exhibit time correlation \cite{Hasselmann1976Theory1, Zwanzig1961Memory, Vasseur2004ColorEnvironmentalNoise}. We investigate two types of red noise influences on PDEs that are losing stability. These are (i) additive noise in the dynamical equations and (ii) linear boundary additive noise. We analytically examine linearised SPDEs corresponding to non-linear bifurcating SPDEs \cite{bernuzzi2023bifurcations, Bernuzzi2024EWSSPDEBoundary, Gowda2015EWSSPDE}. The characteristics of such systems are relevant to real-world systems because the noise is assumed to be small, and the dynamics evolve close to equilibrium, where linear first-order behaviour dominates. In numerical experiments, we then confirm the predicted behaviour.

Our findings indicate that several SPDE components need to be considered to properly understand the occurring destabilisation. The nature of the noise, i.e., the regions and the modes it affects, has complex implications for the solution. However, with respect to CSD, these implications turn out to be negligible in a certain sense: a generic observable of the system will exhibit an increase in variance before the bifurcation. At the same time, a generic observable of a system without a bifurcation but with increasing noise correlation time will exhibit false EWSs. Such indications could be misinterpreted as an approaching bifurcation. When considering the structure of the deterministic PDE, the spectrum of the linearised equation plays a two-fold role: it characterises the linear stability of the studied equilibrium and small gaps between the eigenvalues delay EWSs as the critical transition is approached. This effect is worse for the limit of a purely continuous spectrum, discussed further below. Conversely, the signals can be highly affected by enhanced perturbations along specific modes generated by the interplay between additive noise on the component itself and noise mediated by other generalized eigenfunctions in the same generalized eigenspace.

The paper is structured as follows. We introduce standard mathematical tools employed in the construction of the EWSs in Section \ref{sec: Prel}. In Section \ref{sec: MathResults}, we prove fundamental statements about the limit behaviour of system variance in linear SPDEs losing stability and driven by red noise. Such a stability loss is a generic feature of dynamical bifurcations in non-linear systems. Both the case of stability loss under stationary noise conditions and under non-stationary noise conditions are investigated. We then confirm these findings in numerical models of practical relevance in Section \ref{sec: NumResults}. Finally, we discuss the results in the context of ongoing research on EWSs for bifurcations, such as climate tipping points.

\section{Preliminaries}\label{sec: Prel}
The following framework allows for the modelling of space-time dynamical systems under the influence of stochastic perturbations. More concretely, we analyse the time evolution of a physical quantity $u$ defined on a space domain under the influence of so-called red noise. We set the space domain $\mathcal{X}_1\subset\mathbb{R}^N$ and the Hilbert space $\mathcal{H}_1:=L^2(\mathcal{X}_1)$ of possible solutions $u(\cdot,t)\in \mathcal{H}_1$ for any $t>0$. We indicate the scalar product of this solution space as $\left\langle \cdot, \cdot \right\rangle$. In contrast, we refer to the scalar product of any other Hilbert space $\tilde{\mathcal{H}}$ as $\left\langle \cdot, \cdot \right\rangle_{\tilde{\mathcal{H}}}$. The scalar product with respect to specific probing functions will play a central role in defining the concept of system variance, a potential EWS of bifurcations. We focus on SPDEs on $\mathcal{X}_1$ with red noise or boundary red noise. As such, we define the boundary $\mathcal{X}_0:=\partial\mathcal{X}_1$ of $\mathcal{X}_1$, and we label $\mathcal{H}_0:=L^2(\mathcal{X}_0)$. For $\kappa > 0$, $\sigma > 0$ and $j\in\{0,1\}$, we define the Ornstein-Uhlenbeck process $\xi_j=\xi_j(x,t)$ that solves
\begin{align} \label{eq:syst_xi}
    \text{d}\xi_j(x,t) &= -\kappa \xi_j(x,t) \text{d}t + \sigma Q_j^\frac{1}{2} \text{d}W_t^j ,
\end{align}
for any $x\in\mathcal{X}_j$ and $t>0$. Such a process is often referred to in the literature as red noise \cite{Hanggi1994ColoredNoiseDynamicalSys, Hanggi1993RedNoiseResonance, Miguel1981MultiplicativeRedNoiseEst, Liu2023RedNoiseOcean, Newman1997RedNoiseExtratropicalFlow}. This is because its power spectral density is weighted most heavily in the low (red) frequencies. The constant $\kappa>0$ controls the characteristic correlation time $1/\kappa$ of the noise. The smaller $\kappa$, the longer it takes for correlations in the noise to decay. For $j\in\{0,1\}$, i.e., either boundary or domain noise, the noise term is composed as follows: $Q_j$ is a positive self-adjoint operator in $\mathcal{H}_j$ with real eigenvalues $\left\{q_i\right\}_{i\in\mathbb{N}_{>0}}$ that are bounded from below by $c>0$ and corresponding eigenbasis $\left\{b_i\right\}_{i\in\mathbb{N}_{>0}}$ of $\mathcal{H}_j$. $Q_j$ can be thought of as attributing varying finite noise amplitudes to all modes $b_i$ on the domain of interest. The stochastic part itself is the cylindrical Wiener process $W_t^j$, which can be written as
\begin{align*}
    W_t^j
    = W^j(x,t)
    = \sum_{n=1}^\infty b_n^j(x) \beta_n(t) ,
\end{align*}
for the family of independent scalar Wiener processes $\left\{\beta_n\right\}_{n\in\mathbb{N}_{>0}}$. Its differential is then to be interpreted as
\begin{align*}
    \text{d} W_t^j
    = \sum_{n=1}^\infty b_n^j(x) \text{d} \beta_n(t) .
\end{align*}
The noise is induced in $\xi_j$ by a $Q_j$-Wiener process \cite{da2014stochastic}. Consequently, for every fixed $x\in\mathcal{X}_j$, $\xi_j(x,t)$ is a regular one-dimensional Ornstein-Uhlenbeck process. The interdependencies of processes $\xi_j(x_1,t)$ and $\xi_j(x_2,t)$ is determined by $Q_j$.
\smallskip

We study three SPDEs and their respective associated mild solution $u=u(x,t)$. A mild solution solves the integral form \cite[Theorem 5.4]{da2014stochastic} corresponding to the SPDE for a specific choice of probability space. The existence and uniqueness of the solution for each case are discussed in \cite{Bernuzzi2024EWSSPDEBoundary, Bernuzzi2024EWSSPDEContinuousSpectrum}. The processes $\xi_0$ and $\xi_1$ are considered to be perturbations in the system that define $u$. Their intensity is indicated by $\sigma_{\text{R}}>0$.
\begin{enumerate}
    \item[(a)] First, we consider $u^{\text{d}}=u^{\text{d}}(x,t)$ that solves the following SPDE with domain noise $\xi_1$,
    \begin{align} \label{eq:case_1}
        \left\{ \begin{alignedat}{2}
            \text{d}u^{\text{d}}(x,t) &= \left( A_0(p) u^{\text{d}}(x,t) + \sigma_{\text{R}} \xi_1(x,t) \right) \text{d}t , \\
            u^{\text{d}}(x,0) &= u_0(x) \in \mathcal{H}_1 ,
        \end{alignedat} \right.
    \end{align}
    for any $x\in\mathcal{X}_1$ and $t>0$. The purely-discrete, linear operator $A_0(p)$ is assumed to be negative for $p<0$ with eigenvalues 
    \begin{align*}
        0
        >\operatorname{Re} \left( \lambda^{(p)}_1 \right) 
        >\operatorname{Re} \left( \lambda^{(p)}_2 \right) 
        \geq \operatorname{Re} \left( \lambda^{(p)}_3 \right) 
        \geq \dots ,
    \end{align*}
    which are assumed to be continuous in $p$. This is a free parameter that describes exogenous changes to the system. Equation \ref{eq:case_1} constitutes a deterministically linearly stable system. A loss of linear stability occurs whenever at least one eigenvalue crosses the imaginary axis. This is a generic occurrence in bifurcating dynamical systems. We assume that the eigenvalue $\lambda^{(p)}_1$ is the only one to reach the imaginary axis at $p=0$. For simplicity throughout the paper, we consider that $-\kappa$ is not in the spectrum of $A_0(p)$ for any small $p\leq 0$ and fixed $\kappa>0$. We denote as $\operatorname{R}$ the resolvent of an invertible operator. That is $\operatorname{R}(\lambda, B)=(B-\lambda)^{-1}$, for $\lambda\in\mathbb{C}$ that is not in the spectrum of the operator $B$. For any $p\leq 0$, we denote by $A_0(p)^*$ the adjoint operator of $A_0(p)$ with respect to the scalar product on $\mathcal{H}_1$. We write $\overline{z}$ for the conjugate of $z\in\mathbb{C}$. The operators $A_0(p)$ and $A_0(p)^*$ are assumed to be closed and densely defined in $\mathcal{H}_1$. We use $m_a$ to indicate the algebraic multiplicity of an eigenvalue. For simplicity, their geometric multiplicity is set to $1$. For any $i\in\mathbb{N}_{>0}$, $p\leq 0$ and $k\in\left\{1,\dots,m_a\left(\lambda_i^{(p)}\right)\right\}$, the generalized eigenfunctions of $A_0(p)$ and $A_0(p)^*$  corresponding respectively to $\lambda_i^{(p)}$ and $\overline{\lambda_i^{(p)}}$ are labeled as $e_{i,k}^{(p)}$ and ${e_{i,k}^{(p)}}^*$ and satisfy the Jordan block structure
    \begin{align} \label{eq:def_gen_eig}
    \begin{alignedat}{3}
        & A_0(p) e_{i,k}^{(p)} = \lambda_i^{(p)} e_{i,k}^{(p)},
        && \quad A_0(p)^* {e_{i,k}^{(p)}}^* = \overline{\lambda_i^{(p)}} {e_{i,k}^{(p)}}^*,
        &&\quad \text{for $k=1$},\\
        & A_0(p) e_{i,k}^{(p)} = \lambda_i^{(p)} e_{i,k}^{(p)} + e_{i,k-1}^{(p)},
        && \quad A_0(p)^* {e_{i,k}^{(p)}}^* = \overline{\lambda_i^{(p)}} {e_{i,k}^{(p)}}^* + {e_{i,k-1}^{(p)}}^*,
        &&\quad \text{for $k \neq 1$}.
    \end{alignedat}
    \end{align}
    We assume that such functions are continuous in $\mathcal{H}_1$ with regard to $p$. The deterministically invariant subspaces generated by the generalized eigenfunctions of $A_0(p)$ and $A_0(p)^*$ associated to the eigenvalue $\lambda_i^{(p)}$ and $\overline{\lambda_i^{(p)}}$ are denoted respectively as $E_i(p)$ and $E_i(p)^*$. Their dimension is labeled as $M_i=m_a\left(\lambda_i^{(p)}\right)$ and is assumed to be independent of $p$. For each $i\in\mathbb{N}_{>0}$, the sets $\left\{e_{i,k}^{(p)}\right\}_{k\in\left\{1,\dots,M_i\right\}}$ and $\left\{{e_{i,M_i-k+1}^{(p)}}^*\right\}_{k\in\left\{1,\dots,M_i\right\}}$ are scaled to form a biorthogonal system. Each family is assumed to be complete \cite{Bernuzzi2024EWSSPDEBoundary,zhang2001completeness} in $\mathcal{H}_1$. For $i\in\mathbb{N}_{>0}$, we label $e_{i}^{(p)} = e_{i,1}^{(p)}$ and ${e_{i}^{(p)}}^* = {e_{i,1}^{(p)}}^*$ if $M_i=1$. Lastly, we set $e_{i,0}^{(p)} = {e_{i,0}^{(p)}}^* \equiv 0$ for all $i\in \mathbb{N}_{>0}$.
    Such a definition of generalized eigenfunctions of $A_0(p)^*$ implies the construction of the functions
        \begin{align*}                      \operatorname{R}\left(A_0(p)^*+\kappa\right) {e_{i,k}^{(p)}}^* = 
            - 
            \sum_{j=1}^{k} \left(-\overline{\lambda_i^{(p)}}-\kappa\right)^{-k+j-1} {e_{i,j}^{(p)}}^*=:\mu_{i,k}^{(p,\kappa)}
        \end{align*}
    for any $i\in\mathbb{N}_{>0}$ and $k\in\{1,\dots,M_i\}$, which are used in the theorems to follow.

    To summarize, the function $u^d$ solves a linear equation under the influence of red noise that is added at every point $x$ within the space domain $\mathcal{X}_1$. The linear spectrum is assumed to be discrete so that we can find a set of basis functions (modes) $\left\{{e_{i,k}^{(p)}}^*\right\}$ of $\mathcal{H}_1$ that separates neatly into generalized eigenspaces associated with the discrete eigenvalues. The eigenspace $E_1(p)^*$ is of particular interest since it is associated with the critical eigenvalue $\lambda^{(p)}_1$, which will cross the imaginary axis, e.g., during a bifurcation. These modes will experience the most direct system destabilization.
    \smallskip
    
    \item[(b)] As a second case, we consider $u^{\text{c}}=u^{\text{c}}(x,t)$ to solve
    \begin{align} \label{eq:case_2}
        \left\{ \begin{alignedat}{2}
            \text{d}u^{\text{c}}(x,t) &= \left( f(x,p) u^{\text{c}}(x,t) + \sigma_{\text{R}} \xi_1(x,t) \right) \text{d}t , \\
            u^{\text{c}}(0,x) &= u_0(x) \in \mathcal{H}_1 ,
        \end{alignedat} \right.
    \end{align}
    for $x\in\mathcal{X}_1$ and $t>0$. The function $f:\mathcal{X}_1\times\mathbb{R}_{<0}\to\mathbb{R}_{<0}$ is assumed to be analytic. For a fixed $x_\ast$, the function satisfies
    \begin{align*}
        f(x,p) < 0 \quad \text{and} \quad f(x_\ast,0) = 0 ,
    \end{align*}
    for any $(x,p)\in\mathcal{X}_1\times\mathbb{R}_{\leq0}\setminus\{(x_\ast,0)\}$. Similarly to the previous case, we consider for simplicity values of $p<0<\kappa$ such that 
    \begin{align*}
        f(x,p)+\kappa \neq 0
    \end{align*}
    for any $x\in\mathcal{X}_1$. In contrast, the operator $f$ can have a continuous spectrum. The implications of this on the presence of EWS is studied in Section \ref{sec: MathResults}.
    \smallskip
    
    \item[(c)] Lastly, we observe the solution $u^{\text{b}}=u^{\text{b}}(x,t)$ of
    \begin{align} \label{eq:case_3}
        \left\{ \begin{alignedat}{2}
            \text{d}u^{\text{b}}(x,t) &= A(p) u^{\text{b}}(x,t) \text{d}t , \\
            u^{\text{b}}(0,x) &= u_0(x) \in \mathcal{H}_1 ,
        \end{alignedat} \right.
    \end{align}
    for $x\in\mathcal{X}_1$ and
    \begin{align*}
        \gamma(p) u^{\text{b}}(x,t) &= \sigma_{\text{R}} \xi_0(x,t) ,
    \end{align*}
    on the boundary $x\in\mathcal{X}_0$ and $t>0$. The deterministic part of \eqref{eq:case_3} has the same properties as that of the first considered case. In this case, we are investigating the effect of setting noise on the boundary of the space domain. The linear operator
    \begin{align*}
        \gamma(p):\mathcal{D}\left( \gamma(p) \right) \subseteq \mathcal{H}_1 \to \mathcal{H}_0
    \end{align*}
    defines the boundary conditions. Furthermore, we assume that, for fixed $p\leq 0$, there exists a continuous $q=q(p)\in\mathbb{R}$ such that for any boundary value problem
    \begin{equation*}
        (A(p)- q)w=0 \quad , \quad \gamma(p)\; w=v\;,
    \end{equation*}
    with $v\in \mathcal{H}_0$, there exists a unique solution $w=D(p)v\in\mathcal{D}(A(p))\subseteq \mathcal{H}_1$. For any $p\leq 0$, we indicate $D(p)^*$ as the adjoint operator of $D(p)$ with respect to the scalar products on the Hilbert spaces $\mathcal{H}_1$ and $\mathcal{H}_0$. We assume the operator $D(p)^*$ to be uniformly bounded in $L^2(\mathcal{H}_1;\mathcal{H}_0)$ 
    %and the inverse of $D(p) Q_0 D(p)^*$ to uniformly bounded in $L^2(\mathcal{H}_1;\mathcal{H}_1)$ 
    for any $p\leq 0$. Setting 
    \begin{align*}
        A_0(p) v = A(p) v
    \end{align*}
    for any $v\in\mathcal{H}_1$ such that $\gamma(p) v = 0$, we assume $A_0$ to satisfy the properties described above in $(a)$. Moreover, we consider values of $q$ that are not in the spectrum of $A_0(p)$ and  $A_0(p)^*$ for $p$ close to $0$. We obtain then that 
    \begin{align} \label{eq:def_Lambda}
        \Lambda(p):=\left(A_0(p)- q\right) D(p)
        Q_0
        D(p)^* \left(A_0(p)^*- q\right) 
    \end{align}
    depends on operator $\gamma$. The mild solution of the last system is defined \cite{da1993evolution} in the form
    \begin{align*}
        u^{\text{b}}(x,t) = \text{e}^{A_0(p) t} u_0(x) + \sigma \sigma_{\text{R}} \int_0^t \text{e}^{A_0(p) (t-s)} \left( A_0(p) - q \right) D(p) \int_0^s \text{e}^{-\kappa (s-r)} Q_0^\frac{1}{2} \text{d}W^0_r \; \text{d}s .
    \end{align*}
\end{enumerate}
Following \cite{bernuzzi2023bifurcations,Gowda2015EWSSPDE,Romano2019ScalLawsAndEWSs}, we aim to construct early-warning signs to the approaches $p\to 0^-$ and $\kappa\to 0^+$, respectively. On such thresholds, the dissipativity in the linear system that defines $u$ and $\xi$ is lost, and the origin in $\mathcal{H}_1\times\mathcal{H}_1$ is not a stable deterministic equilibrium. However, the two limit cases have very different physical interpretations. While the $p\to 0^-$ limit is a stand-in for dynamical bifurcations of non-linear systems, $\kappa\to 0^+$ represents a change in the characteristics of the driving noise. In the context of, e.g., climate tipping points, only the former limit would be of interest. We hope to find EWS in the system variance with respect to different probing functions. If such a probing function is composed of a destabilising mode, we would conventionally expect variance to increase. We define the covariance as $\text{Cov}$. In case $(a)$ and $(b)$, we set the linear variance operator $V_t:\mathcal{H}_1\times \mathcal{H}_1 \to \mathcal{H}_1\times \mathcal{H}_1$, such that
\begin{align*}
    \left\langle \begin{pmatrix}
        v_1 \\ v_2
    \end{pmatrix} , V_t \begin{pmatrix}
        w_1 \\ w_2
    \end{pmatrix} \right\rangle_{\mathcal{H}_1\times\mathcal{H}_1}
    = \text{Cov} \left( \left\langle u(\cdot,t), v_1 \right\rangle
    + \left\langle \xi_1(\cdot,t), v_2 \right\rangle, 
    \left\langle u(\cdot,t), w_1 \right\rangle 
    + \left\langle \xi_1(\cdot,t), w_2 \right\rangle \right) ,
\end{align*}
for $v_1,v_2,w_1,w_2\in\mathcal{H}_1$. Since we are solely interested in the variance on variable $u$, we will set $v_2=w_2\equiv 0$. Furthermore, in case $(c)$, the boundary noise requires the different definition $V_t^{\text{b}}:\mathcal{H}_1\times \mathcal{H}_0 \to \mathcal{H}_1\times \mathcal{H}_0$ and
\begin{align*}
    \left\langle \begin{pmatrix}
        v_1 \\ v_2
    \end{pmatrix} , V_t^{\text{b}} \begin{pmatrix}
        w_1 \\ w_2
    \end{pmatrix} \right\rangle_{\mathcal{H}_1\times\mathcal{H}_0}
    = \text{Cov} \left( \left\langle u(\cdot,t), v_1 \right\rangle
    + \left\langle \xi_0(\cdot,t), v_2 \right\rangle_{\mathcal{H}_0}, 
    \left\langle u(\cdot,t), w_1 \right\rangle 
    + \left\langle \xi_0(\cdot,t), w_2 \right\rangle_{\mathcal{H}_0} \right) ,
\end{align*}
for $v_1,w_1\in\mathcal{H}_1$ and $v_2,w_2\in\mathcal{H}_0$. We construct the time-asymptotic variance operator as $V_\infty = \underset{t\to \infty}{\text{lim}} V_t$ and $V_\infty^{\text{b}} = \underset{t\to \infty}{\text{lim}} V_t^{\text{b}}$. Such an observable is employed in the next section to construct EWSs. These are defined as its rate of divergence in the limits $p\to 0^-$ and $\kappa\to 0^+$. As such, we refer to the Landau notation \cite{Bernuzzi2024EWSSPDEContinuousSpectrum}
\begin{align*}
    r_1(p,\kappa) = \Theta_p \left( r_2(p,\kappa) \right) 
    \quad \Longleftrightarrow \quad
    \underset{p \to 0^-}{\text{lim}} \frac{r_1(p,\kappa)}{r_2(p,\kappa)} \in (0,+\infty), \quad
    &r_1(p,\kappa) = \mathcal{O}_p \left( r_2(p,\kappa) \right) 
    \quad \Longleftrightarrow \quad
    \underset{p \to 0^-}{\text{lim}} \frac{r_1(p,\kappa)}{r_2(p,\kappa)} \in [0,+\infty), \\
    r_1(p,\kappa) = \Theta_\kappa \left( r_2(p,\kappa) \right) 
    \quad \Longleftrightarrow \quad
    \underset{\kappa \to 0^+}{\text{lim}} \frac{r_1(p,\kappa)}{r_2(p,\kappa)} \in (0,+\infty) , \quad
    &r_1(p,\kappa) = \mathcal{O}_\kappa \left( r_2(p,\kappa) \right) 
    \quad \Longleftrightarrow \quad
    \underset{\kappa \to 0^+}{\text{lim}} \frac{r_1(p,\kappa)}{r_2(p,\kappa)} \in [0,+\infty) , \\
    r_1(p,\kappa) = \Theta \left( r_2(p,\kappa) \right) 
    \quad \Longleftrightarrow \quad
    \underset{(p,\kappa) \to (0,0)}{\text{lim}} \frac{r_1(p,\kappa)}{r_2(p,\kappa)} \in (0,+\infty) , \quad
    &r_1(p,\kappa) = \mathcal{O} \left( r_2(p,\kappa) \right) 
    \quad \Longleftrightarrow \quad
    \underset{(p,\kappa) \to (0,0)}{\text{lim}} \frac{r_1(p,\kappa)}{r_2(p,\kappa)} \in [0,+\infty) , 
\end{align*}
for any pair of locally continuous functions $r_1: \mathbb{R}_{<0}\times\mathbb{R}_{>0} \to \mathbb{R}_{>0}$ and $r_2: \mathbb{R}_{<0}\times\mathbb{R}_{>0} \to \mathbb{R}_{>0}$. In essence, the $\Theta$ equivalence is a stronger asymptotic characteristic than the standard $\mathcal{O}$ equivalence, since it implies boundedness of the ratio and its inverse.

\section{Main Results}\label{sec: MathResults}

In this section, we prove the scaling law of the time-asymptotic variance of the mild solutions associated with the SPDEs \eqref{eq:case_1}, \eqref{eq:case_2} and \eqref{eq:case_3}. This is considered in the limits $p\to 0^-$ and $\kappa\to 0^+$, where the dissipativity of the models is lost. In the case of $p\to 0^-$, we thus discover an EWS of linear stability loss. The scaling w.r.t.~$\kappa\to 0^+$ on the other hand should be considered a false EWS, since no genuine destabilisation of the deterministic dynamics took place.

\subsection{Discrete Spectrum}

We first consider $u^{\text{d}}=u^{\text{d}}(x,t)$ that solves \eqref{eq:case_1} and would like to make statements about the variance. The linear drift term in the system has a purely discrete spectrum. As a result, the scaling law of the time-asymptotic variance, i.e., the rate of its convergence or divergence, depends on the functions along which it is observed. The following theorem indicates these sensible modes and the corresponding asymptotics.

\begin{thm} \label{thm:1}
We consider $u^{\text{d}}=u^{\text{d}}(x,t)$ that solves
    \begin{align*}
        \left\{ \begin{alignedat}{2}
            \text{d}u^{\text{d}}(x,t) &= \left( A_0(p) u^{\text{d}}(x,t) + \sigma_{\text{R}} \xi_1(x,t) \right) \text{d}t , \\
            \text{d}\xi_1(x,t) &= -\kappa \xi_1(x,t) + \sigma Q_1^\frac{1}{2} \text{d}W_t^1 ,
        \end{alignedat} \right.
    \end{align*}
with initial conditions in $\mathcal{H}_1$, $x\in\mathcal{X}_1$, $p<0$ and $t>0$. Then, the scaling laws
    \begin{align*}
        \left\lvert \left\langle \begin{pmatrix}
            {e_{i_1,k_1}^{(p)}}^* \\ 0
        \end{pmatrix}, V_\infty \begin{pmatrix}
            {e_{i_2,k_2}^{(p)}}^* \\ 0
        \end{pmatrix} \right\rangle_{\mathcal{H}_1\times \mathcal{H}_1} \right\rvert
        &= \Theta_\kappa \left( \kappa^{-1} \right) 
        \quad \text{for any } \quad p<0
    \end{align*}
and
    \begin{align*}
        \left\lvert \left\langle \begin{pmatrix}
            {e_{i_1,k_1}^{(p)}}^* \\ 0
        \end{pmatrix}, V_\infty \begin{pmatrix}
            {e_{i_2,k_2}^{(p)}}^* \\ 0
        \end{pmatrix} \right\rangle_{\mathcal{H}_1\times \mathcal{H}_1} \right\rvert
        &= \Theta_p \left( \left\lvert \overline{\lambda_{i_1}^{(p)}}+\lambda_{i_2}^{(p)} \right\rvert^{-(k_1+k_2-1)} \right)
        \quad \text{for any} \quad \kappa>0
    \end{align*}
hold for any $i_1,i_2\in\mathbb{N}_{>0}$, $k_1\in\{1,\dots,M_{i_1}\}$ and $k_2\in\{1,\dots,M_{i_2}\}$.
\end{thm}

The theorem implies that as an eigenvalue $\lambda_1^{(p)}$ crosses the imaginary axis, i.e., $\text{Re}\left(\lambda_{1}^{(p)}\right)\to 0^-$, the variance evaluated w.r.t.~a probing function $v$ will diverge, as long as it is at least partly aligned with an associated generalized eigenfunction ${e^{(p)}_{1,k}}^*$. This can be considered to be the generic case since the alternative is only true for a restrictive set of functions (see Corollary \ref{cor:2} below). The biorthogonality of the generalized eigenfunctions of $A_0(p)$ and $A_0(p)^*$ implies that the rate of divergence directly corresponds to the location of the (generalized) eigenfunction $e^{(p)}_{1,M_1-k+1}$ within the Jordan block associated with $\lambda_1$. In particular, the closer $e^{(p)}_{1,M_1-k+1}$ is to being a true eigenfunction in that Jordan block, i.e., $k$ is close to $M_1$, the faster the divergence. An analytic description of the Jordan block structure is provided at the end of the section.
\begin{proof}
We define the operator
    \begin{align} \label{eq:B0_a}
        B_0(p)=\begin{pmatrix}
            A_0(p) & \sigma_{\text{R}} \\ 0 & -\kappa
        \end{pmatrix}
    \end{align}
and its adjoint in respect to $\mathcal{H}_1\times\mathcal{H}_1$
    \begin{align*}
        B_0(p)^*=\begin{pmatrix}
            A_0(p)^* & 0 \\ \sigma_{\text{R}} & -\kappa
        \end{pmatrix} .
    \end{align*}
These matrices allow to combine the two equations in~\eqref{eq:case_1} into one linear equation. They generate the $C_0$-semigroups
    \begin{align*}
        \text{e}^{B_0(p) t}=\begin{pmatrix}
            \text{e}^{A_0(p) t} & \sigma_{\text{R}} \left( \text{e}^{A_0(p) t} - \text{e}^{-\kappa t} \right) \operatorname{R}\left(A_0(p)+\kappa\right) \\ 0 & \text{e}^{-\kappa t}
        \end{pmatrix}
    \end{align*}
and
    \begin{align*}
        {\text{e}^{B_0(p) t}}^* = \text{e}^{B_0(p)^* t}=\begin{pmatrix}
            \text{e}^{A_0(p)^* t} & 0 \\ \sigma_{\text{R}} \operatorname{R}\left(A_0(p)^*+\kappa\right) \left( \text{e}^{A_0(p)^* t} - \text{e}^{-\kappa t} \right) & \text{e}^{-\kappa t}
        \end{pmatrix}
    \end{align*}
for $t>0$, respectively. The time-asymptotic variance operator is then obtained by applying It\^o's isometry to the mild solution formula \cite[Theorem 5.2]{da2014stochastic} and is given by
    \begin{align*}
        V_\infty
        &=\int_0^\infty \text{e}^{B_0(p) t} \begin{pmatrix}
            0 & 0 \\ 0 & \sigma^2 Q_1
        \end{pmatrix} {\text{e}^{B_0(p) t}}^* \text{d}t\\
        &= \int_0^\infty \begin{pmatrix}
            \text{e}^{A_0(p) t} & \sigma_{\text{R}} \left( \text{e}^{A_0(p) t} - \text{e}^{-\kappa t} \right) \operatorname{R}\left(A_0(p)+\kappa\right) \\ 0 & \text{e}^{-\kappa t}
        \end{pmatrix}
        \begin{pmatrix}
            0 & 0 \\ 0 & \sigma^2 Q_1
        \end{pmatrix} 
        \begin{pmatrix}
            \text{e}^{A_0(p)^* t} & 0 \\ \sigma_{\text{R}} \operatorname{R}\left(A_0(p)^*+\kappa\right) \left( \text{e}^{A_0(p)^* t} - \text{e}^{-\kappa t} \right) & \text{e}^{-\kappa t}
        \end{pmatrix} \text{d}t \\
        &= \sigma^2 \int_0^\infty \begin{pmatrix}
            \sigma_{\text{R}}^2 \left( \text{e}^{A_0(p) t} - \text{e}^{-\kappa t} \right) \operatorname{R}\left(A_0(p)+\kappa\right) Q_1 \operatorname{R}\left(A_0(p)^*+\kappa\right) \left( \text{e}^{A_0(p)^* t} - \text{e}^{-\kappa t} \right) & 
            \sigma_{\text{R}} \left( \text{e}^{A_0(p)-\kappa t} - \text{e}^{-2\kappa t} \right) \operatorname{R}\left(A_0(p)+\kappa\right) Q_1 \\ 
            \sigma_{\text{R}} Q_1 \operatorname{R}\left(A_0(p)^*+\kappa\right) \left( \text{e}^{A_0(p)^*-\kappa t} - \text{e}^{-2\kappa t} \right) & 
            \text{e}^{-2\kappa t} Q_1
        \end{pmatrix} \text{d}t .
    \end{align*}
In the next steps, we employ
    \begin{align*}
        \text{e}^{A_0(p)^* t} {e_{i,k}^{(p)}}^*(x)
        = \text{e}^{\overline{\lambda_{i}^{(p)}} t} \sum_{j=1}^{k} \frac{t^{k-j}}{(k-j)!} {e_{i,j}^{(p)}}^*(x) .
    \end{align*}
Setting $i_1,i_2\in\mathbb{N}_{>0}$, $k_1\in\{1,\dots,M_{i_1}\}$ and $k_2\in\{1,\dots,M_{i_2}\}$, this entails that
    \begin{align} \label{eq:main_formula_a}
        &\left\langle \begin{pmatrix}
            {e_{i_1,k_1}^{(p)}}^* \\ 0
        \end{pmatrix}, V_\infty \begin{pmatrix}
            {e_{i_2,k_2}^{(p)}}^* \\ 0
        \end{pmatrix} \right\rangle_{\mathcal{H}_1\times \mathcal{H}_1} \nonumber\\
        = &\sigma^2 \int_0^\infty \left\langle {e_{i_1,k_1}^{(p)}}^*, \sigma_{\text{R}}^2 \left( \text{e}^{A_0(p) t} - \text{e}^{-\kappa t} \right) \operatorname{R}\left(A_0(p)+\kappa\right) Q_1 \operatorname{R}\left(A_0(p)^*+\kappa\right) \left( \text{e}^{A_0(p)^* t} - \text{e}^{-\kappa t} \right) {e_{i_2,k_2}^{(p)}}^* \right\rangle \text{d}t \nonumber\\
        %= &\sigma^2 \sigma_{\text{R}}^2 \int_0^\infty \Bigg\langle \operatorname{R}\left(A_0(p)^*+\kappa\right) \left( \text{e}^{\overline{\lambda_{i_1}^{(p)}} t} \sum_{j_1=1}^{k_1} \frac{t^{k_1-j_1}}{(k_1-j_1)!} {e_{i_1,j_1}^{(p)}}^* 
        %- \text{e}^{-\kappa t} {e_{i_1,k_1}^{(p)}}^* \right) , 
        %Q_1 %\operatorname{R}\left(A_0(p)^*+\kappa\right) \left( \text{e}^{\overline{\lambda_{i_2}^{(p)}} t} \sum_{j_2=1}^{k_2} \frac{t^{k_2-j_2}}{(k_2-j_2)!} {e_{i_2,j_2}^{(p)}}^* 
        %- \text{e}^{-\kappa t} {e_{i_2,k_2}^{(p)}}^* \right) \Bigg\rangle \text{d}t \nonumber\\
        = &\sigma^2 \sigma_{\text{R}}^2 \int_0^\infty \Bigg\langle \text{e}^{\overline{\lambda_{i_1}^{(p)}} t} \sum_{j_1=1}^{k_1} \frac{t^{k_1-j_1}}{(k_1-j_1)!} \mu_{i_1,j_1}^{(p,\kappa)} 
        - \text{e}^{-\kappa t} \mu_{i_1,k_1}^{(p,\kappa)}, 
        Q_1 \Bigg( \text{e}^{\overline{\lambda_{i_2}^{(p)}} t} \sum_{j_2=1}^{k_2} \frac{t^{k_2-j_2}}{(k_2-j_2)!} \mu_{i_2,j_2}^{(p,\kappa)}
        - \text{e}^{-\kappa t} \mu_{i_2,k_2}^{(p,\kappa)} \Bigg) \Bigg\rangle \text{d}t \\
        = &\sigma^2 \sigma_{\text{R}}^2 \Bigg( \sum_{j_1=1}^{k_1} \sum_{j_2=1}^{k_2} \begin{pmatrix}
            k_1-j_1+k_2-j_2 \\ k_1-j_1
        \end{pmatrix} \left(-\overline{\lambda_{i_1}^{(p)}}-\lambda_{i_2}^{(p)}\right)^{-k_1+j_1-k_2+j_2-1}
        \left\langle \mu_{i_1,j_1}^{(p,\kappa)} ,
        Q_1 \mu_{i_2,j_2}^{(p,\kappa)} \right\rangle \nonumber\\
        &- \sum_{j_2=1}^{k_2} \left(-\lambda_{i_2}^{(p)}+\kappa\right)^{-k_2+j_2-1}
        \left\langle \mu_{i_1,k_1}^{(p,\kappa)} , 
        Q_1 \mu_{i_2,j_2}^{(p,\kappa)} \right\rangle
        - \sum_{j_1=1}^{k_1} \left(-\overline{\lambda_{i_1}^{(p)}}+\kappa\right)^{-k_1+j_1-1}
        \left\langle \mu_{i_1,j_1}^{(p,\kappa)} ,
        Q_1 \mu_{i_2,k_2}^{(p,\kappa)} \right\rangle
        + (2 \kappa)^{-1} \left\langle \mu_{i_1,k_1}^{(p,\kappa)} , 
        Q_1 \mu_{i_2,k_2}^{(p,\kappa)} \right\rangle \Bigg) , \nonumber
    \end{align}
which is the covariance of $u^{\text{d}}$ along the modes ${e_{i_1,k_1}^{(p)}}^*$ and ${e_{i_2,k_2}^{(p)}}^*$. Since $-\kappa$ is not in the spectrum of $A_0(p)^*$ and at most one term in the sum diverges in the limits $p\to 0^-$ and $\kappa\to 0^+$, it follows that the scaling laws are
    \begin{align*}
        \left\lvert \left\langle \begin{pmatrix}
            {e_{i_1,k_1}^{(p)}}^* \\ 0
        \end{pmatrix}, V_\infty \begin{pmatrix}
            {e_{i_2,k_2}^{(p)}}^* \\ 0
        \end{pmatrix} \right\rangle_{\mathcal{H}_1\times \mathcal{H}_1} \right\rvert
        &= \Theta_p \left( \left\lvert \overline{\lambda_{i_1}^{(p)}}+\lambda_{i_2}^{(p)} \right\rvert^{-(k_1+k_2-1)} \right)
    \end{align*}
and
    \begin{align*}
        \left\lvert \left\langle \begin{pmatrix}
            {e_{i_1,k_1}^{(p)}}^* \\ 0
        \end{pmatrix}, V_\infty \begin{pmatrix}
            {e_{i_2,k_2}^{(p)}}^* \\ 0
        \end{pmatrix} \right\rangle_{\mathcal{H}_1\times \mathcal{H}_1} \right\rvert
        &= \Theta_\kappa \left( \kappa^{-1} \right) .
    \end{align*}
\end{proof}

In Theorem \ref{thm:1}, the divergence of system variance is associated with specific observables, i.e., probing functions. In the limit $p\to 0^-$, this divergence occurs only for $i_1=i_2=1$ for construction. The fact that the generalized eigenfunctions of $A_0(p)^*$ are complete in $\mathcal{H}_1$ for any $p<0$ enables the extension of the EWS to a set of functions dense in $\mathcal{H}_1$. This makes the EWS a generic occurrence.

\begin{cor} \label{cor:2}
We consider $u^{\text{d}}=u^{\text{d}}(x,t)$, the mild solution of
    \begin{align*}
        \left\{ \begin{alignedat}{2}
            \text{d}u^{\text{d}}(x,t) &= \left( A_0(p) u^{\text{d}}(x,t) + \sigma_{\text{R}} \xi_1(x,t) \right) \text{d}t , \\
            \text{d}\xi_1(x,t) &= -\kappa \xi_1(x,t) + \sigma Q_1^\frac{1}{2} \text{d}W_t^1 .
        \end{alignedat} \right.
    \end{align*}
with initial conditions in $\mathcal{H}_1$, $x\in\mathcal{X}_1$, $p<0$ and $t>0$. For $M\in\mathbb{N}_{>0}$, we set $h_1^{(p)},h_2^{(p)}\in\underset{i=1}{\overset{M}{\bigoplus}} E_i(p)^* \setminus \underset{i=2}{\overset{M}{\bigoplus}} E_i(p)^*\subset \mathcal{H}_1$. Then 
    \begin{align*}
        \left\lvert \left\langle \begin{pmatrix}
            h_1^{(p)} \\ 0
        \end{pmatrix}, V_\infty \begin{pmatrix}
            h_2^{(p)} \\ 0
        \end{pmatrix} \right\rangle_{\mathcal{H}_1\times \mathcal{H}_1} \right\rvert
        &= \Theta_\kappa \left( \kappa^{-1} \right)
        \quad \text{for any} \quad p<0
    \end{align*}
holds. Furthermore, if $h_1^{(p)}$ and $h_2^{(p)}$ satisfy
\begin{align} \label{eq:condition_cor_1}
    a_{1,M_1,1}:=\left\langle h_1^{(p)}, e_{1,1}^{(p)} \right\rangle \neq 0 \neq \left\langle h_2^{(p)}, e_{1,1}^{(p)} \right\rangle=:a_{1,M_1,2}
\end{align}
for any $p\leq 0$, then
    \begin{align*}
        \left\lvert \left\langle \begin{pmatrix}
            h_1^{(p)} \\ 0
        \end{pmatrix}, V_\infty \begin{pmatrix}
            h_2^{(p)} \\ 0
        \end{pmatrix} \right\rangle_{\mathcal{H}_1\times \mathcal{H}_1} \right\rvert
        &= \Theta_p \left( \text{Re} \left( - \lambda_{1}^{(p)} \right)^{-(2 M_1-1)} \right)
        \quad \text{for any} \quad \kappa>0
    \end{align*}
holds.
\end{cor}
\begin{proof}
We define the families $\left\{ a_{i,k,1}^{(p)} \right\}\subset\mathbb{C}$ and $\left\{ a_{i,k,2}^{(p)} \right\}\subset\mathbb{C}$ for $i\in\{1,\dots,M\}$ and $k\in\{1,\dots,M_i\}$, such that
\begin{align*}
    h_1^{(p)} = \underset{k\in\{1,\dots,M_i\}}{\sum_{i\in\{1,\dots,M\}}} a_{i,k,1}^{(p)} {e_{i,k}^{(p)}}^*
    \qquad \text{and} \qquad
    h_2^{(p)} = \underset{k\in\{1,\dots,M_i\}}{\sum_{i\in\{1,\dots,M\}}} a_{i,k,2}^{(p)} {e_{i,k}^{(p)}}^*
\end{align*}
for any $p\leq 0$. It follows that
\begin{align*}
    \left\langle \begin{pmatrix}
        h_1^{(p)} \\ 0
    \end{pmatrix}, V_\infty \begin{pmatrix}
        h_2^{(p)} \\ 0
    \end{pmatrix} \right\rangle_{\mathcal{H}_1\times \mathcal{H}_1}
    = \underset{k_1\in\{1,\dots,M_{i_1}\}}{\sum_{i_1\in\{1,\dots,M\}}} 
    \underset{k_2\in\{1,\dots,M_{i_2}\}}{\sum_{i_2\in\{1,\dots,M\}}}
    a_{i_1,k_1,1}^{(p)}
    \overline{a_{i_2,k_2,2}^{(p)}}
    \left\langle \begin{pmatrix}
        {e_{i_1,k_1}^{(p)}}^* \\ 0
    \end{pmatrix}, V_\infty \begin{pmatrix}
        {e_{i_2,k_2}^{(p)}}^* \\ 0
    \end{pmatrix} \right\rangle_{\mathcal{H}_1\times \mathcal{H}_1} .
\end{align*}
From the form \eqref{eq:main_formula_a} in Theorem \ref{thm:1}, this implies that
\begin{align*}
    \left\lvert \left\langle \begin{pmatrix}
        h_1^{(p)} \\ 0
    \end{pmatrix}, V_\infty \begin{pmatrix}
        h_2^{(p)} \\ 0
    \end{pmatrix} \right\rangle_{\mathcal{H}_1\times \mathcal{H}_1} \right\rvert
    &= \Theta_\kappa \left( \left\lvert (2 \kappa)^{-1} \underset{k_1\in\{1,\dots,M_{i_1}\}}{\sum_{i_1\in\{1,\dots,M\}}} 
    \underset{k_2\in\{1,\dots,M_{i_2}\}}{\sum_{i_2\in\{1,\dots,M\}}} 
    a_{i_1,k_1,1}^{(p)}
    \overline{a_{i_2,k_2,2}^{(p)}} 
    \left\langle  \mu_{i_1,k_1}^{(p,\kappa)} , 
    Q_1 \mu_{i_2,k_2}^{(p,\kappa)} \right\rangle \right\rvert \right)\\
    &= \Theta_\kappa \left( \kappa^{-1} \left\lvert \left\langle \underset{k_1\in\{1,\dots,M_{i_1}\}}{\sum_{i_1\in\{1,\dots,M\}}} a_{i_1,k_1,1}^{(p)} \mu_{i_1,k_1}^{(p,\kappa)} , 
    Q_1 \underset{k_2\in\{1,\dots,M_{i_2}\}}{\sum_{i_2\in\{1,\dots,M\}}} a_{i_2,k_2,2}^{(p)} \mu_{i_2,k_2}^{(p,\kappa)} \right\rangle \right\rvert \right)\\
    &= \Theta_\kappa \left( \kappa^{-1} \right) .
\end{align*}
In the limit $p\to 0^-$, the variance
    \begin{align*}
        \left\lvert \left\langle \begin{pmatrix}
            {e_{i_1,k_1}^{(p)}}^* \\ 0
        \end{pmatrix}, V_\infty \begin{pmatrix}
            {e_{i_2,k_2}^{(p)}}^* \\ 0
        \end{pmatrix} \right\rangle_{\mathcal{H}_1\times \mathcal{H}_1} \right\rvert
        &= \Theta_p \left( \left\lvert \overline{\lambda_{i_1}^{(p)}}+\lambda_{i_2}^{(p)} \right\rvert^{-(k_1+k_2-1)} \right)
    \end{align*}
diverges only for $i_1=i_2=1$. Furthermore, its scaling law is defined by the choice of $k_1,k_2\in\{1,\dots,M_1\}$, and the highest rate of divergence is associated with $k_1=k_2=M_1$. Equation \eqref{eq:main_formula_a} and condition \eqref{eq:condition_cor_1} imply that
\begin{align*}
    \left\lvert \left\langle \begin{pmatrix}
        h_1^{(p)} \\ 0
    \end{pmatrix}, V_\infty \begin{pmatrix}
        h_2^{(p)} \\ 0
    \end{pmatrix} \right\rangle_{\mathcal{H}_1\times \mathcal{H}_1} \right\rvert
    &= \Theta_p \left( \left\lvert \underset{k_1\in\{1,\dots,M_{i_1}\}}{\sum_{i_1\in\{1,\dots,M\}}} 
    \underset{k_2\in\{1,\dots,M_{i_2}\}}{\sum_{i_2\in\{1,\dots,M\}}} 
    a_{i_1,k_1,1}^{(p)}
    \overline{a_{i_2,k_2,2}^{(p)}} 
    \left\langle \begin{pmatrix}
        {e_{i_1,k_1}^{(p)}}^* \\ 0
    \end{pmatrix}, V_\infty \begin{pmatrix}
        {e_{i_2,k_2}^{(p)}}^* \\ 0
    \end{pmatrix} \right\rangle_{\mathcal{H}_1\times \mathcal{H}_1} \right\rvert \right) \\
    &= \Theta_p \left( \left\lvert 
    \sum_{k_1=1}^{M_1} 
    \sum_{k_2=1}^{M_1} 
    a_{1,k_1,1}^{(p)}
    \overline{a_{1,k_2,2}^{(p)}} 
    \left\langle \begin{pmatrix}
        {e_{1,k_1}^{(p)}}^* \\ 0
    \end{pmatrix}, V_\infty \begin{pmatrix}
        {e_{1,k_2}^{(p)}}^* \\ 0
    \end{pmatrix} \right\rangle_{\mathcal{H}_1\times \mathcal{H}_1} \right\rvert \right) \\
    &= \Theta_p \left( \left\lvert 
    \sum_{k_1=1}^{M_1} 
    \sum_{k_2=1}^{M_1} 
    a_{1,k_1,1}^{(p)}
    \overline{a_{1,k_2,2}^{(p)}} 
    \text{Re} \left( - \lambda_{1}^{(p)} \right)^{-(k_1+k_2-1)} \right\rvert \right) \\
    &= \Theta_p \left(  
    \text{Re} \left( - \lambda_{1}^{(p)} \right)^{-(2 M_1 -1)} \right) .
\end{align*}
\end{proof}

The statement of Corollary \ref{cor:2} justifies the observation of the time-asymptotic variance along a large family of functions. As an example, any appropriate approximation of an indicator function is a suitable direction along which the EWS can be studied \cite{Bernuzzi2024EWSSPDEBoundary}. Equivalently, the data variance on a large time interval in a hand-picked region of space is likely to display the highest rate of divergence depending on the model.

\subsection{Continuous Spectrum}

We study $u^{\text{c}}=u^{\text{c}}(x,t)$, the mild solution of \eqref{eq:case_2} for any $x\in\mathcal{X}_1$, $p<0$ and $t>0$. Since the spectrum of the considered linear drift operator is not discrete, the observation of the time-asymptotic variance along favored modes is not viable. Hence, we search for other functions in $\mathcal{H}_1$ that enable the construction of the EWS.

\begin{thm} \label{thm:3}
We consider $u^{\text{c}}=u^{\text{c}}(x,t)$, the mild solution of
    \begin{align*}
        \left\{ \begin{alignedat}{2}
            \text{d}u^{\text{c}}(x,t) &= \left( f(x,p) u^{\text{c}}(x,t) + \sigma_{\text{R}} \xi_1(x,t) \right) \text{d}t , \\
            \text{d}\xi_1(x,t) &= -\kappa \xi_1(x,t) + \sigma Q_1^\frac{1}{2} \text{d}W_t^1 ,
        \end{alignedat} \right.
    \end{align*}
with initial conditions in $\mathcal{H}_1$, $x\in\mathcal{X}_1$, $p<0$ and $t>0$. For any $g_1,g_2 \in \mathcal{H}_1$, it holds
    \begin{align} \label{eq:main_formula_b}
        \left\langle \begin{pmatrix}
            g_1 \\ 0
        \end{pmatrix}, V_\infty \begin{pmatrix}
            g_2 \\ 0
        \end{pmatrix} \right\rangle_{\mathcal{H}_1\times \mathcal{H}_1}
        &= \sigma^2 \sigma_{\text{R}}^2 \int_0^\infty \left\langle \frac{ \text{e}^{f(\cdot,p) t} - \text{e}^{-\kappa t} }{f(\cdot,p) + \kappa} g_1, Q_1 \frac{ \text{e}^{f(\cdot,p) t} - \text{e}^{-\kappa t} }{f(\cdot,p) + \kappa} g_2 \right\rangle \text{d}t .
    \end{align}
\end{thm}
\begin{proof}
We again define the operators
    \begin{align} \label{eq:B0_b}
        B_0(p)=\begin{pmatrix}
            f(\cdot,p) & \sigma_{\text{R}} \\ 0 & -\kappa
        \end{pmatrix}
    \end{align}
and its adjoint in $\mathcal{H}_1\times\mathcal{H}_1$,
    \begin{align*}
        B_0(p)^*=\begin{pmatrix}
            f(\cdot,p) & 0 \\ \sigma_{\text{R}} & -\kappa
        \end{pmatrix}
    \end{align*}
to combine the two equations in~\eqref{eq:case_2}. They generate the $C_0$-semigroups in $t$,
    \begin{align*}
        \text{e}^{B_0(p) t}=\begin{pmatrix}
            \text{e}^{f(\cdot,p) t} & \sigma_{\text{R}} \frac{ \text{e}^{f(\cdot,p) t} - \text{e}^{-\kappa t} }{f(\cdot,p) + \kappa} \\ 0 & \text{e}^{-\kappa t}
        \end{pmatrix}
    \end{align*}
and
    \begin{align*}
        {\text{e}^{B_0(p) t}}^* = \text{e}^{B_0(p)^* t}=\begin{pmatrix}
            \text{e}^{f(\cdot,p) t} & 0 \\ \sigma_{\text{R}} \frac{ \text{e}^{f(\cdot,p) t} - \text{e}^{-\kappa t} }{f(\cdot,p) + \kappa} & \text{e}^{-\kappa t}
        \end{pmatrix} ,
    \end{align*}
respectively. It follows from the construction of the covariance operator \cite{Bernuzzi2024EWSSPDEContinuousSpectrum,da1993evolution} that
    \begin{align*}
        V_\infty
        &= \int_0^\infty \text{e}^{B_0(p) t} \begin{pmatrix}
            0 & 0 \\ 0 & \sigma^2 Q_1
        \end{pmatrix} {\text{e}^{B_0(p) t}}^* \text{d}t\\
        &= \int_0^\infty \begin{pmatrix}
            \text{e}^{f(\cdot,p) t} & \sigma_{\text{R}} \frac{ \text{e}^{f(\cdot,p) t} - \text{e}^{-\kappa t} }{f(\cdot,p) + \kappa} \\ 0 & \text{e}^{-\kappa t}
        \end{pmatrix}
        \begin{pmatrix}
            0 & 0 \\ 0 & \sigma^2 Q_1
        \end{pmatrix} 
        \begin{pmatrix}
            \text{e}^{f(\cdot,p) t} & 0 \\ \sigma_{\text{R}} \frac{ \text{e}^{f(\cdot,p) t} - \text{e}^{-\kappa t} }{f(\cdot,p) + \kappa} & \text{e}^{-\kappa t}
        \end{pmatrix} \text{d}t \\
        &= \sigma^2 \int_0^\infty \begin{pmatrix}
            \sigma_{\text{R}}^2 \frac{ \text{e}^{f(\cdot,p) t} - \text{e}^{-\kappa t} }{f(\cdot,p) + \kappa} Q_1 \frac{ \text{e}^{f(\cdot,p) t} - \text{e}^{-\kappa t} }{f(\cdot,p) + \kappa} & 
            \sigma_{\text{R}} \frac{ \text{e}^{f(\cdot,p) t} - \text{e}^{-\kappa t} }{f(\cdot,p) + \kappa} Q_1 \\ 
            \sigma_{\text{R}} Q_1 \frac{ \text{e}^{f(\cdot,p) t} - \text{e}^{-\kappa t} }{f(\cdot,p) + \kappa} & 
            \text{e}^{-2\kappa t} Q_1
        \end{pmatrix} \text{d}t .
    \end{align*}
Setting $g_1,g_2\in\mathcal{H}_1$, the time-asymptotic covariance along those functions is then
    \begin{align*}
        \left\langle \begin{pmatrix}
            g_1 \\ 0
        \end{pmatrix}, V_\infty \begin{pmatrix}
            g_2 \\ 0
        \end{pmatrix} \right\rangle_{\mathcal{H}_1\times \mathcal{H}_1}
        &= \sigma^2 \sigma_{\text{R}}^2 \int_0^\infty \left\langle \frac{ \text{e}^{f(\cdot,p) t} - \text{e}^{-\kappa t} }{f(\cdot,p) + \kappa} g_1, Q_1 \frac{ \text{e}^{f(\cdot,p) t} - \text{e}^{-\kappa t} }{f(\cdot,p) + \kappa} g_2 \right\rangle \text{d}t ,
    \end{align*}
for $p<0$.
\end{proof}

We set $\alpha>0$ and we also consider the case $\mathcal{X}_1\subset\mathcal{R}$. We then focus on $f(x,p)=-\lvert x \rvert^\alpha + p$ for any $x\in\mathbb{R}$, $p\leq 0$. This type of function enables further construction of EWS in case of $f$ being analytic. We define $g_1=g_2=\mathbbm{1}_\mathcal{S}$, the indicator function on the Lebesgue-measurable set $\mathcal{S}\subset \mathcal{X}_1$ and assume $0=x_\ast\in\mathcal{S}$. The following corollary makes the results about system variance in the above theorem concrete. These results are encouraging in the sense that taking an indicator function $\mathbbm{1}_\mathcal{S}$ of a Lebesgue-measurable set in $\mathcal{X}_1$ as a probing function, we always observe the EWS of diverging variance.

\begin{cor} \label{cor:4}
\begin{enumerate}
    \item[(a)] We study $u^{\text{c}}=u^{\text{c}}(x,t)$, the mild solution of 
    \begin{align*}
        \left\{ \begin{alignedat}{2}
            \text{d}u^{\text{c}}(x,t) &= \left( \left(-\lvert x \rvert^\alpha + p \right) u^{\text{c}}(x,t) + \sigma_{\text{R}} \xi_1(x,t) \right) \text{d}t , \\
            \text{d}\xi_1(x,t) &= -\kappa \xi_1(x,t) + \sigma Q_1^\frac{1}{2} \text{d}W_t^1 ,
        \end{alignedat} \right.
    \end{align*}
with initial conditions in $\mathcal{H}_1$, $x\in\mathcal{X}_1$, $\alpha>0$, $p<0$ and $t>0$. Then, for any $g=\mathbbm{1}_\mathcal{S}$, the scaling law of the time-asymptotic variance along $g$ is given by
    \begin{align*}
        \left\langle \begin{pmatrix}
            g \\ 0
        \end{pmatrix}, V_\infty \begin{pmatrix}
            g \\ 0
        \end{pmatrix} \right\rangle_{\mathcal{H}_1\times \mathcal{H}_1}
        &= \Theta_\kappa \left( \kappa^{-1} \right)
    \end{align*}
for the limit $\kappa\to 0^+$. It also entails that, for $p\to 0^-$, the following holds:
    \begin{align*}
        \begin{alignedat}{3}
            &\mathllap{\bullet} \qquad \qquad
            \left\langle \begin{pmatrix}
                g \\ 0
            \end{pmatrix}, V_\infty \begin{pmatrix}
                g \\ 0
            \end{pmatrix} \right\rangle_{\mathcal{H}_1\times \mathcal{H}_1}
            &&= \Theta_p \left( (-p)^{-1+\frac{1}{\alpha}} \right) ,
            \quad &&\text{for} \quad \alpha>1 ; \\
            &\mathllap{\bullet} \qquad \qquad
            \left\langle \begin{pmatrix}
                g \\ 0
            \end{pmatrix}, V_\infty \begin{pmatrix}
                g \\ 0
            \end{pmatrix} \right\rangle_{\mathcal{H}_1\times \mathcal{H}_1}
            &&= \Theta_p \left( \text{log}(-p) \right) ,
            \quad &&\text{for} \quad \alpha=1 ; \\
            &\mathllap{\bullet} \qquad \qquad
            \left\langle \begin{pmatrix}
                g \\ 0
            \end{pmatrix}, V_\infty \begin{pmatrix}
                g \\ 0
            \end{pmatrix} \right\rangle_{\mathcal{H}_1\times \mathcal{H}_1}
            &&= \Theta_p \left( 1 \right) ,
            \quad &&\text{for} \quad 0<\alpha<1 .
        \end{alignedat}
    \end{align*}
    \item[(b)] We consider $u^{\text{c}}=u^{\text{c}}(x,t)$, the mild solution of 
    \begin{align*}
        \left\{ \begin{alignedat}{2}
            \text{d}u^{\text{c}}(x,t) &= \left( \left( f_{\text{an}}(x) + p \right) u^{\text{c}}(x,t) + \sigma_{\text{R}} \xi_1(x,t) \right) \text{d}t , \\
            \text{d}\xi_1(x,t) &= -\kappa \xi_1(x,t) + \sigma Q_1^\frac{1}{2} \text{d}W_t^1 ,
        \end{alignedat} \right.
    \end{align*}
with initial conditions in $\mathcal{H}_1$, $x\in\mathcal{X}_1$, $p<0$ and $t>0$. We assume that 
    \begin{align*}
        f_{\text{an}}(x)= \sum_{n=1}^\infty a_n x^n
    \end{align*}
for any $x\in\mathcal{X}_1$ and for the family $\left\{a_n\right\}_{n\in\mathbb{N}_{>0}}\subset\mathbb{R}$. It follows that for any $g=\mathbbm{1}_\mathcal{S}$, the scaling law of the time-asymptotic variance along $g$ is
    \begin{align*}
        \left\langle \begin{pmatrix}
            g \\ 0
        \end{pmatrix}, V_\infty \begin{pmatrix}
            g \\ 0
        \end{pmatrix} \right\rangle_{\mathcal{H}_1\times \mathcal{H}_1}
        &= \Theta_\kappa \left( \kappa^{-1} \right)
    \end{align*}
for the limit $\kappa\to 0^+$.
Moreover, we fix $n_\ast$ such that
    \begin{align*}
        n_\ast=\underset{n\in\mathbb{N}_{>0}}{\text{argmin}} \left\{ a_n\neq 0 \right\} .
    \end{align*}
Then, for any $g=\mathbbm{1}_\mathcal{S}$, the rate of divergence of the time-asymptotic variance along $g$ for the limit $p\to 0^-$ is given by
    \begin{align*}
        \left\langle \begin{pmatrix}
            g \\ 0
        \end{pmatrix}, V_\infty \begin{pmatrix}
            g \\ 0
        \end{pmatrix} \right\rangle_{\mathcal{H}_1\times \mathcal{H}_1}
        &= \Theta_p \left( (-p)^{-1+\frac{1}{n_\ast}} \right) , \quad \text{if} \quad n_\ast>1 ,
    \end{align*}
    or 
    \begin{align*}
        \left\langle \begin{pmatrix}
            g \\ 0
        \end{pmatrix}, V_\infty \begin{pmatrix}
            g \\ 0
        \end{pmatrix} \right\rangle_{\mathcal{H}_1\times \mathcal{H}_1}
        &= \Theta_p \left( \text{log}(-p) \right) , \quad \text{if} \quad n_\ast=1 .
    \end{align*}
\end{enumerate}
\end{cor}
\begin{proof}
We consider the $u^{\text{c}}$ as assumed in Theorem \ref{thm:3}. The assumption of $Q_1$ bounded and bounded from below far from zero and the formula \eqref{eq:main_formula_b} imply that the rates of divergence of the time-asymptotic variance are equivalent to the integral
    \begin{align*}
        &\int_0^\infty \left\langle \frac{ \text{e}^{f(\cdot,p) t} - \text{e}^{-\kappa t} }{f(\cdot,p) + \kappa} g, \frac{ \text{e}^{f(\cdot,p) t} - \text{e}^{-\kappa t} }{f(\cdot,p) + \kappa} g \right\rangle \text{d}t
        = \int_0^\infty \int_\mathbb{R} \left( \frac{ \text{e}^{f(x,p) t} - \text{e}^{-\kappa t} }{f(x,p) + \kappa} g \right)^2 \text{d}x \; \text{d}t \\
        = &\int_0^\infty \int_\mathcal{S} \left( \frac{ \text{e}^{f(x,p) t} - \text{e}^{-\kappa t} }{f(x,p) + \kappa} \right)^2 \text{d}x \; \text{d}t 
        = \int_\mathcal{S} \left( - \frac{1}{2 f(x,p)} - \frac{2}{f(x,p) - \kappa} + \frac{1}{2 \kappa} \right) \frac{ 1 }{\left( f(x,p) + \kappa \right)^2} \text{d}x .
    \end{align*}
From the negative sign of the analytic function $f$ for any $p<0$, it follows that
    \begin{align*}
        \left\langle \begin{pmatrix}
            g \\ 0
        \end{pmatrix}, V_\infty \begin{pmatrix}
            g \\ 0
        \end{pmatrix} \right\rangle_{\mathcal{H}_1\times \mathcal{H}_1}
        = \Theta_\kappa \left( \kappa^{-1} \right) .
    \end{align*}
In part $(a)$ we assume $f(x,p)=-\lvert x \rvert^\alpha + p$ for any $x\in\mathcal{X}_1$ and $p<0$. As such, it entails that
    \begin{align} \label{eq:rate_cont}
        \left\langle \begin{pmatrix}
            g \\ 0
        \end{pmatrix}, V_\infty \begin{pmatrix}
            g \\ 0
        \end{pmatrix} \right\rangle_{\mathcal{H}_1\times \mathcal{H}_1}
        = \Theta_p \left( -\int_\mathcal{S} \frac{1}{f(x,p)} \text{d}x \right)
        = \Theta_p \left( \int_\mathcal{S} \frac{1}{\lvert x \rvert^\alpha - p} \text{d}x \right).
    \end{align}
The rate depends on the parameter $\alpha$ and can be obtained as described in \cite[Theorem 3.1]{Bernuzzi2024EWSSPDEContinuousSpectrum}.
From the construction of $n_\ast$, it follows that there exists $c>0$ such that
\begin{align*}
    c x^{n_\ast} \leq f_{\text{an}}(x) \leq c^{-1} x^{x_\ast}
\end{align*}
for $x$ in a neighbourhood of $x_\ast=0$. Then, the rate in the limit $p\to 0^-$ described in \eqref{eq:rate_cont} entails the remainder of statement $(b)$.
\end{proof}

Corollary \ref{cor:4} provides EWS to different models. An example of a system that is known to display linear drift with continuous spectrum is the Swift-Hohenberg equation \cite{burke2007homoclinic,burke2007snakes,kao2014spatial} on $\mathcal{X}_1=\mathbb{R}$, known to find application in the study of electricity fields in crystal optical fiber resonator \cite{hariz2019swift}, or on $\mathcal{X}_1=\mathbb{R}^2$, for which a generalization has been applied in optics \cite{lega1994swift}. The proposed EWS thus enable the prediction of a bifurcation upon the presence of red noise in the model. The results of Corollary \ref{cor:4} are also relevant in the case of a purely discrete spectrum with small gaps between eigenvalues. In fact, for the Ginzburg-Landau equation on a large interval, the inclusion of red noise to represent minor perturbations in the system enables the construction of EWS that do not recognize the discreteness of the spectrum until $p$ is in the proximity of $0$. Consequently, the EWS are damped prior to the bifurcation threshold. Such a model finds applications in phase-ordering kinetic \cite{bray2002theory}, quantum mechanics \cite{faris1982large} and climate science \cite{hogele2011metastability}.

\subsection{Boundary Noise}

We observe the behaviour of $u^{\text{b}}=u^{\text{b}}(x,t)$, the mild solution of \eqref{eq:case_3}, in the limits $p\to 0^-$ and $\kappa\to 0^+$. As in the previous subsections, we explore the scaling law of the time-asymptotic variance as an EWS. Since the linear operator associated with the drift term in \eqref{eq:case_3} has a purely discrete spectrum, the time-asymptotic variance can be studied as an observable along favoured modes. Yet, the structure of the noise requires a different approach to its construction in comparison to the other examples in the section. The difference lies in the fact that the noise is now present in the boundary conditions, rather than directly in the PDE itself. Still, the following theorem states that the generic divergence of system variance is recovered just as in the first considered case.

\begin{thm} \label{thm:5}
We consider $u^{\text{b}}=u^{\text{b}}(x,t)$, mild solution of
    \begin{align*}
        \left\{ \begin{alignedat}{2}
            \text{d}u^{\text{b}}(x,t) &= A(p) u^{\text{b}}(x,t) \text{d}t , \\
            \gamma(p) u^{\text{b}}(x,t) &= \sigma_{\text{R}} \xi_0(x,t) , \\ 
            \text{d}\xi_0(x,t) &= -\kappa \xi_0(x,t) + \sigma Q_0^\frac{1}{2} \text{d}W_t^0 ,
        \end{alignedat} \right.
    \end{align*}
with initial conditions in $\mathcal{H}_1$, $x\in\mathcal{X}_1$, $p<0$ and $t>0$. Then, the scaling laws
    \begin{align*}
        \left\lvert \left\langle \begin{pmatrix}
            {e_{i_1,k_1}^{(p)}}^* \\ 0
        \end{pmatrix}, V_\infty^{\text{b}} \begin{pmatrix}
            {e_{i_2,k_2}^{(p)}}^* \\ 0
        \end{pmatrix} \right\rangle_{\mathcal{H}_1\times \mathcal{H}_0} \right\rvert
        = \mathcal{O}_\kappa \left( \kappa^{-1} \right) \quad \text{for any} \quad p<0
    \end{align*}
and
    \begin{align*}
        \left\lvert \left\langle \begin{pmatrix}
            {e_{i_1,k_1}^{(p)}}^* \\ 0
        \end{pmatrix}, V_\infty^{\text{b}} \begin{pmatrix}
            {e_{i_2,k_2}^{(p)}}^* \\ 0
        \end{pmatrix} \right\rangle_{\mathcal{H}_1\times \mathcal{H}_0} \right\rvert
        = \mathcal{O}_p \left( \left\lvert \overline{\lambda_{i_1}^{(p)}}+\lambda_{i_2}^{(p)} \right\rvert^{-(k_1+k_2-1)} \right) \quad \text{for any} \quad \kappa>0
    \end{align*}
hold for any $i_1,i_2\in\mathbb{N}_{>0}$, $k_1\in \left\{1,\dots,M_{i_1} \right\}$ and $k_2\in \left\{1,\dots,M_{i_2} \right\}$.
\end{thm}
\begin{proof}
We define the operator
    \begin{align} \label{eq:B0_c}
        B_0(p)=\begin{pmatrix}
            A_0(p) & \sigma_{\text{R}} \left(A_0(p)- q\right) D(p) \\ 0 & -\kappa
        \end{pmatrix}
    \end{align}
and its adjoint with respect to $\mathcal{H}_1\times\mathcal{H}_0$
    \begin{align*}
        B_0(p)^*=\begin{pmatrix}
            A_0(p)^* & 0 \\ \sigma_{\text{R}} D(p)^* \left(A_0(p)^*- q\right) & -\kappa
        \end{pmatrix} .
    \end{align*}
Since the second term in the diagonal of $B_0(p)$ is a multiplication operator by a scalar, they generate the $C_0$-semigroups
    \begin{align*}
        \text{e}^{B_0(p) t}=\begin{pmatrix}
            \text{e}^{A_0(p) t} & \sigma_{\text{R}} \left( \text{e}^{A_0(p) t} - \text{e}^{-\kappa t} \right) \operatorname{R}\left(A_0(p)+\kappa\right) \left(A_0(p)- q\right) D(p) \\ 0 & \text{e}^{-\kappa t}
        \end{pmatrix}
    \end{align*}
and
    \begin{align*}
        {\text{e}^{B_0(p) t}}^* = \text{e}^{B_0(p)^* t}=\begin{pmatrix}
            \text{e}^{A_0(p)^* t} & 0 \\ \sigma_{\text{R}} D(p)^* \left(A_0(p)^*- q\right) \operatorname{R}\left(A_0(p)^*+\kappa\right) \left( \text{e}^{A_0(p)^* t} - \text{e}^{-\kappa t} \right) & \text{e}^{-\kappa t}
        \end{pmatrix}
    \end{align*}
for $t>0$, respectively. The time-asymptotic variance operator is then
    \begin{align*}
        V_\infty^{\text{b}}
        &=\int_0^\infty \text{e}^{B_0(p) t} \begin{pmatrix}
            0 & 0 \\ 0 & \sigma^2 Q_0
        \end{pmatrix} {\text{e}^{B_0(p) t}}^* \text{d}t
    \end{align*}
and the first element in the integrand corresponds to
    \begin{align*}
        &\sigma^2 \sigma_{\text{R}}^2
        \left( \text{e}^{A_0(p) t} - \text{e}^{-\kappa t} \right) \operatorname{R}\left(A_0(p)+\kappa\right) \left(A_0(p)- q\right) D(p)
        Q_0
        D(p)^* \left(A_0(p)^*- q\right) \operatorname{R}\left(A_0(p)^*+\kappa\right) \left( \text{e}^{A_0(p)^* t} - \text{e}^{-\kappa t} \right) \\
        =&\sigma^2 \sigma_{\text{R}}^2
        \left( \text{e}^{A_0(p) t} - \text{e}^{-\kappa t} \right) \operatorname{R}\left(A_0(p)+\kappa\right) \Lambda(p) \operatorname{R}\left(A_0(p)^*+\kappa\right) \left( \text{e}^{A_0(p)^* t} - \text{e}^{-\kappa t} \right),
    \end{align*}
    for $\Lambda(p)$ defined in \eqref{eq:def_Lambda}.
Following equivalent steps to \eqref{eq:main_formula_a} we obtain
    \begin{align} \label{eq:main_formula_c}
        &\left\langle \begin{pmatrix}
            {e_{i_1,k_1}^{(p)}}^* \\ 0
        \end{pmatrix}, V_\infty^{\text{b}} \begin{pmatrix}
            {e_{i_2,k_2}^{(p)}}^* \\ 0
        \end{pmatrix} \right\rangle_{\mathcal{H}_1\times \mathcal{H}_0} \nonumber\\
        = &\sigma^2 \sigma_{\text{R}}^2 \int_0^\infty \Bigg\langle \text{e}^{\overline{\lambda_{i_1}^{(p)}} t} \sum_{j_1=1}^{k_1} \frac{t^{k_1-j_1}}{(k_1-j_1)!} \mu_{i_1,j_1}^{(p,\kappa)} 
        - \text{e}^{-\kappa t} \mu_{i_1,k_1}^{(p,\kappa)}, 
         \Lambda(p) \Bigg( \text{e}^{\overline{\lambda_{i_2}^{(p)}} t} \sum_{j_2=1}^{k_2} \frac{t^{k_2-j_2}}{(k_2-j_2)!} \mu_{i_2,j_2}^{(p,\kappa)}
        - \text{e}^{-\kappa t} \mu_{i_2,k_2}^{(p,\kappa)} \Bigg) \Bigg\rangle \text{d}t \\
        = &\sigma^2 \sigma_{\text{R}}^2 \Bigg( \sum_{j_1=1}^{k_1} \sum_{j_2=1}^{k_2} \begin{pmatrix}
            k_1-j_1+k_2-j_2 \\ k_1-j_1
        \end{pmatrix} \left(-\overline{\lambda_{i_1}^{(p)}}-\lambda_{i_2}^{(p)}\right)^{-k_1+j_1-k_2+j_2-1}
        \left\langle \mu_{i_1,j_1}^{(p,\kappa)} ,
        \Lambda(p) \mu_{i_2,j_2}^{(p,\kappa)} \right\rangle \nonumber\\
        &- \sum_{j_2=1}^{k_2} \left(-\lambda_{i_2}^{(p)}+\kappa\right)^{-k_2+j_2-1}
        \left\langle \mu_{i_1,k_1}^{(p,\kappa)} , 
        \Lambda(p) \mu_{i_2,j_2}^{(p,\kappa)} \right\rangle
        - \sum_{j_1=1}^{k_1} \left(-\overline{\lambda_{i_1}^{(p)}}+\kappa\right)^{-k_1+j_1-1}
        \left\langle \mu_{i_1,j_1}^{(p,\kappa)} ,
        \Lambda(p) \mu_{i_2,k_2}^{(p,\kappa)} \right\rangle
        + (2 \kappa)^{-1} \left\langle \mu_{i_1,k_1}^{(p,\kappa)} , 
        \Lambda(p) \mu_{i_2,k_2}^{(p,\kappa)} \right\rangle \Bigg) . \nonumber
    \end{align}
We notice that
\begin{align*}
    \left\lvert \lambda_{i}^{(p)} + \kappa \right\rvert = \Theta(1) 
    \qquad \text{and} \qquad
    \left\lvert \lambda_{i}^{(p)} - \kappa \right\rvert = \Theta(1) 
\end{align*}
for any $i\in\mathbb{N}_{>0}$. Furthermore,
\begin{align*}
    \left\lvert \overline{\lambda_{i_1}^{(p)}} + \lambda_{i_2}^{(p)} \right\rvert = \Theta(1) 
\end{align*}
for any $(i_1,i_2)\in\mathbb{N}_{>0}\times\mathbb{N}_{>0}\setminus\{(1,1)\}$. Lastly, since the functions $\mu_{i,k}^{(p,\kappa)}$ are finite combinations of generalized eigenfunctions of $A_0(p)^*$, the property
    \begin{align*}
        A_0(p)^* \mu_{i,j}^{(p,\kappa)} = \overline{\lambda_i^{(p)}} \mu_{i,k}^{(p,\kappa)} - \left(\overline{\lambda_i^{(p)}} + \kappa \right)^{-1} \mu_{i,k-1}^{(p,\kappa)} ,
    \end{align*}
for any $i\in\mathbb{N}_{>0}$, $k\in\{1,\dots,M_i\}$, $p<0$ and $\kappa>0$, and from the uniform boundedness of $D(p)^*$, we obtain that 
\begin{align*}
    \left\langle \mu_{i_1,k_1}^{(p,\kappa)} , 
    \Lambda(p) \mu_{i_2,k_2}^{(p,\kappa)} \right\rangle = \mathcal{O}(1) ,
\end{align*}
for any $(i_1,i_2)\in\mathbb{N}_{>0}\times\mathbb{N}_{>0}\setminus\{(1,1)\}$, $k_1\in\{1,\dots,M_{i_1}\}$ and  $k_2\in\{1,\dots,M_{i_2}\}$.
The EWS are consequently defined by the scaling laws of the observable in \eqref{eq:main_formula_c}. The rate of divergence in $\kappa\to 0^+$ is implied by the last term in \eqref{eq:main_formula_c}; whereas for the limit $p\to 0^-$ it is induced by the behaviour of the term in the first sum in the righthand-side of \eqref{eq:main_formula_c} corresponding to $j_1=j_2=1$. As such, the theorem is proven.
\end{proof}
Theorem \ref{thm:5} describes the rates of the time-asymptotic variance operator along chosen modes. The next corollary extends the use of such an EWS to a larger set of functions much like Corollary \ref{cor:2}. Due to the completeness of the generalized eigenfunctions of $A_0(p)^*$ in $\mathcal{H}_1$ for any $p\leq 0$, such a set is dense in $\mathcal{H}_1$. In this sense, a generic choice of probing function will exhibit the EWS of rising variance also in this case of red noise on the boundary, whose highest possible rate of divergence depends on the boundary conditions of the model. Since the computation of the corresponding map $D(p)$ is not trivial except for particular examples, the degeneracy of the noise \cite{Bernuzzi2024Large, Bernuzzi2024EWSSPDEBoundary} could hinder observation of the stochastic perturbations. As such, an exact scaling law can not be captured without enforcing further assumptions on the boundary map.

\begin{cor} \label{cor:6}
We consider $u^{\text{b}}=u^{\text{b}}(x,t)$, mild solution of
    \begin{align*}
        \left\{ \begin{alignedat}{2}
            \text{d}u^{\text{b}}(x,t) &= A(p) u^{\text{b}}(x,t) \text{d}t , \\
            \gamma(p) u^{\text{b}}(x,t) &= \sigma_{\text{R}} \xi_0(x,t) , \\ 
            \text{d}\xi_0(x,t) &= -\kappa \xi_0(x,t) + \sigma Q_0^\frac{1}{2} \text{d}W_t^0 ,
        \end{alignedat} \right.
    \end{align*}
with initial conditions in $\mathcal{H}_1$, $x\in\mathcal{X}_1$, $p<0$ and $t>0$. For $M\in\mathbb{N}_{>0}$, we set $h_1^{(p)},h_2^{(p)}\in\underset{i=1}{\overset{M}{\bigoplus}} E_i(p)^* \setminus \underset{i=2}{\overset{M}{\bigoplus}} E_i(p)^*\subset \mathcal{H}_1$ such that
\begin{align} \label{eq:condition_cor_3}
    a_{1,M_1,1}:=\left\langle h_1^{(p)}, e_{1,1}^{(p)} \right\rangle \neq 0 \neq \left\langle h_2^{(p)}, e_{1,1}^{(p)} \right\rangle=:a_{1,M_1,2}
\end{align}
for any $p\leq 0$. Then,
    \begin{align*}
        \left\lvert \left\langle \begin{pmatrix}
            h_1^{(p)} \\ 0
        \end{pmatrix}, V_\infty^{\text{b}} \begin{pmatrix}
            h_2^{(p)} \\ 0
        \end{pmatrix} \right\rangle_{\mathcal{H}_1\times \mathcal{H}_0} \right\rvert
        &= \mathcal{O}_\kappa \left( \kappa^{-1} \right) 
        \quad \text{for any} \quad p<0
    \end{align*}
    and
    \begin{align*}
        \left\lvert \left\langle \begin{pmatrix}
            h_1^{(p)} \\ 0
        \end{pmatrix}, V_\infty^{\text{b}} \begin{pmatrix}
            h_2^{(p)} \\ 0
        \end{pmatrix} \right\rangle_{\mathcal{H}_1\times \mathcal{H}_0} \right\rvert
        &= \mathcal{O}_p \left( \text{Re} \left( - \lambda_{1}^{(p)} \right)^{-(2 M_1-1)} \right) 
        \quad \text{for any} \quad \kappa>0
    \end{align*}
hold.
\end{cor}
\begin{proof}
We set the families $\left\{ a_{i,k,1}^{(p)} \right\}\subset\mathbb{C}$ and $\left\{ a_{i,k,2}^{(p)} \right\}\subset\mathbb{C}$ for $i\in\{1,\dots,M\}$ and $k\in\{1,\dots,M_i\}$, such that
\begin{align*}
    h_1^{(p)} = \underset{k\in\{1,\dots,M_i\}}{\sum_{i\in\{1,\dots,M\}}} a_{i,k,1}^{(p)} {e_{i,k}^{(p)}}^*
    \qquad \text{and} \qquad
    h_2^{(p)} = \underset{k\in\{1,\dots,M_i\}}{\sum_{i\in\{1,\dots,M\}}} a_{i,k,2}^{(p)} {e_{i,k}^{(p)}}^* ,
\end{align*}
for any $p\leq 0$. We obtain then
\begin{align*}
    \left\langle \begin{pmatrix}
            h_1^{(p)} \\ 0
        \end{pmatrix}, V_\infty^{\text{b}} \begin{pmatrix}
            h_2^{(p)} \\ 0
        \end{pmatrix} \right\rangle_{\mathcal{H}_1\times \mathcal{H}_0}
    = \underset{k_1\in\{1,\dots,M_{i_1}\}}{\sum_{i_1\in\{1,\dots,M\}}} 
    \underset{k_2\in\{1,\dots,M_{i_2}\}}{\sum_{i_2\in\{1,\dots,M\}}}
    a_{i_1,k_1,1}^{(p)}
    \overline{a_{i_2,k_2,2}^{(p)}}
    \left\langle \begin{pmatrix}
        {e_{i_1,k_1}^{(p)}}^* \\ 0
    \end{pmatrix}, V_\infty^{\text{b}} \begin{pmatrix}
        {e_{i_2,k_2}^{(p)}}^* \\ 0
    \end{pmatrix} \right\rangle_{\mathcal{H}_1\times \mathcal{H}_0} .
\end{align*}
The form \eqref{eq:main_formula_c} in Theorem \ref{thm:5} entails that
\begin{align*}
    \left\lvert \left\langle \begin{pmatrix}
            h_1^{(p)} \\ 0
        \end{pmatrix}, V_\infty^{\text{b}} \begin{pmatrix}
            h_2^{(p)} \\ 0
        \end{pmatrix} \right\rangle_{\mathcal{H}_1\times \mathcal{H}_0} \right\rvert
    &= \Theta_\kappa \left( \left\lvert (2 \kappa)^{-1} \underset{k_1\in\{1,\dots,M_{i_1}\}}{\sum_{i_1\in\{1,\dots,M\}}} 
    \underset{k_2\in\{1,\dots,M_{i_2}\}}{\sum_{i_2\in\{1,\dots,M\}}} 
    a_{i_1,k_1,1}^{(p)}
    \overline{a_{i_2,k_2,2}^{(p)}} 
    \left\langle \mu_{i_1,k_1}^{(p,\kappa)} , 
    \Lambda(p) \mu_{i_2,k_2}^{(p,\kappa)} \right\rangle \right\rvert \right)\\
    &= \Theta_\kappa \left( \kappa^{-1} \left\lvert \left\langle \underset{k_1\in\{1,\dots,M_{i_1}\}}{\sum_{i_1\in\{1,\dots,M\}}} a_{i_1,k_1,1}^{(p)} \mu_{i_1,k_1}^{(p,\kappa)} , 
    \Lambda(p) \underset{k_2\in\{1,\dots,M_{i_2}\}}{\sum_{i_2\in\{1,\dots,M\}}} a_{i_2,k_2,2}^{(p)} \mu_{i_2,k_2}^{(p,\kappa)} \right\rangle \right\rvert \right)\\
    &= \mathcal{O}_\kappa \left( \kappa^{-1} \right) .
\end{align*}
In the limit $p\to 0^-$, the variance
    \begin{align*}
        \left\lvert \left\langle \begin{pmatrix}
        {e_{i_1,k_1}^{(p)}}^* \\ 0
        \end{pmatrix}, V_\infty^{\text{b}} \begin{pmatrix}
        {e_{i_2,k_2}^{(p)}}^* \\ 0
        \end{pmatrix} \right\rangle_{\mathcal{H}_1\times \mathcal{H}_0} \right\rvert
        &= \mathcal{O}_p \left( \left\lvert \overline{\lambda_{i_1}^{(p)}}+\lambda_{i_2}^{(p)} \right\rvert^{-(k_1+k_2-1)} \right)
    \end{align*}
displays divergence only for the indexes $i_1=i_2=1$. Moreover, its rate of divergence is associated with the values $k_1,k_2\in\{1,\dots,M_1\}$. Consequently, equation \eqref{eq:main_formula_c} and condition \eqref{eq:condition_cor_3} imply that
\begin{align*}
    \left\lvert \left\langle \begin{pmatrix}
            h_1^{(p)} \\ 0
        \end{pmatrix}, V_\infty^{\text{b}} \begin{pmatrix}
            h_2^{(p)} \\ 0
        \end{pmatrix} \right\rangle_{\mathcal{H}_1\times \mathcal{H}_0} \right\rvert
    &= \Theta_p \left( \left\lvert \underset{k_1\in\{1,\dots,M_{i_1}\}}{\sum_{i_1\in\{1,\dots,M\}}} 
    \underset{k_2\in\{1,\dots,M_{i_2}\}}{\sum_{i_2\in\{1,\dots,M\}}} 
    a_{i_1,k_1,1}^{(p)}
    \overline{a_{i_2,k_2,2}^{(p)}} 
    \left\langle \begin{pmatrix}
        {e_{i_1,k_1}^{(p)}}^* \\ 0
    \end{pmatrix}, V_\infty^{\text{b}} \begin{pmatrix}
        {e_{i_2,k_2}^{(p)}}^* \\ 0
    \end{pmatrix} \right\rangle_{\mathcal{H}_1\times \mathcal{H}_0} \right\rvert \right) \\
    &= \Theta_p \left( \left\lvert 
    \sum_{k_1=1}^{M_1} 
    \sum_{k_2=1}^{M_1} 
    a_{1,k_1,1}^{(p)}
    \overline{a_{1,k_2,2}^{(p)}} 
    \left\langle \begin{pmatrix}
        {e_{1,k_1}^{(p)}}^* \\ 0
    \end{pmatrix}, V_\infty^{\text{b}} \begin{pmatrix}
        {e_{1,k_2}^{(p)}}^* \\ 0
    \end{pmatrix} \right\rangle_{\mathcal{H}_1\times \mathcal{H}_0} \right\rvert \right) \\
    &= \mathcal{O}_p \left( \left\lvert 
    \sum_{k_1=1}^{M_1} 
    \sum_{k_2=1}^{M_1} 
    a_{1,k_1,1}^{(p)}
    \overline{a_{1,k_2,2}^{(p)}} 
    \text{Re} \left( - \lambda_{1}^{(p)} \right)^{-(k_1+k_2-1)} \right\rvert \right) \\
    &= \mathcal{O}_p \left(  
    \text{Re} \left( - \lambda_{1}^{(p)} \right)^{-(2 M_1 -1)} \right) .
\end{align*}
\end{proof}

\subsection{Considerations on scaling laws and autocorrelation as an EWS}

In the previous subsections, the scaling laws of the time-asymptotic variance along various modes are considered for $p\to 0^-$ and $\kappa\to 0^+$. In such limits, the EWS are shown to display different behaviours depending on the assumption of the spectrum of the linear drift operator in the SPDE that defines $u$. As a concrete example, if the boundary noise does not perturb the solution along the critical mode, the system may not exhibit CSD. The correlation of $u$ and $\xi_j$ for $j\in\{0,1\}$ and, in the case of boundary noise, the degeneracy of the stochastic component imply that the effect of the stochastic perturbation on $\xi_j$ on $u$ is not trivial. Fortunately, the simple form of the linear operator $B_0(p)$ in \eqref{eq:B0_a}, \eqref{eq:B0_b} and \eqref{eq:B0_c} enables the study of the problem along its eigenmodes. We consider first the limit $p\to 0^-$. In \eqref{eq:case_1}, the spectrum of $A_0(p)$ is discrete for any $p\leq 0$ and the EWS displays hyperbolic rate of divergence along the sensible eigenfunction ${e_1^{(p)}}^*$ of its adjoint in Theorem \ref{thm:1}. Furthermore, the EWS along the sensible generalized eigenfunctions ${e_{1,k}^{(p)}}^*$ of $A_0(p)^*$ indicate a faster scaling law depending on their rank $k$. Through the biorthogonality of the generalized eigenfunctions of $A_0(p)$ and $A_0(p)^*$ we know that the the projection of $g\in\mathcal{H}_1$ on ${e_{1,k}^{(p)}}^*$ is equivalent to the coefficient of $g$ on the $e_{1,M_1-k+1}^{(p)}$, for any $k\in \{1,\dots,M_1\}$. This is intended in the sense that for $M\in\mathbb{N}_{>0}$ and
\begin{align*}
    g = \sum_{i=1}^M \sum_{k=1}^{M_i} c_{i,k} e_{i,k}^{(p)} ,
\end{align*}
then it holds
\begin{align*}
    \left\langle g, {e_{i,k}^{(p)}}^* \right\rangle = c_{i,M_i-k+1}
\end{align*}
for any $i\in\{1,\dots,M\}$, $k\in\{1,\dots,M_i\}$ and $p\leq 0$.
As a result, the time-asymptotic covariance along ${e_{1,k_1}^{(p)}}^*$ and ${e_{1,k_2}^{(p)}}^*$ refers to the time-asymptotic covariance of the oscillations of the coefficients of $u^{\text{d}}$, the solution of \eqref{eq:case_1}, along $e_{1,M_1-k_1+1}^{(p)}$ and $e_{1,M_1-k_2+1}^{(p)}$. This observable collects also the oscillations along $e_{1,j_1}^{(p)}$ and $e_{1,j_2}^{(p)}$ for $j_1\in \{M_1-k_1+2,\dots,M_1\}$ and $j_2\in \{M_1-k_2+2,\dots,M_1\}$. This is implied by the fact that oscillations along $e_{1,k}^{(p)}$ imply further perturbations along the mode $e_{1,k-1}^{(p)}$ for any $k\in \{2,\dots,M_1\}$, as shown in \eqref{eq:def_gen_eig}.

The limit in \eqref{eq:case_2} is different in nature from the previous case. The absence of eigenfunctions entails that there are no preferred directions along which the EWS captures the bifurcation. As shown in Corollary \ref{cor:4}, the shape of the spectrum can dampen the EWS and hinder the scaling law of the time-asymptotic variance. Furthermore, for a non-differentiable $f$, the EWS is silenced or assumes a logarithmic rate of divergence. In applications, the first case corresponds to the crossing of the bifurcation threshold being unnoticed by the EWS.

The system \eqref{eq:case_3} displays a similar scaling law to \eqref{eq:case_1}. Nonetheless, the noise perturbation is filtered by $(A_0-q) D(p)$, an operator often unknown in applications. As such, its dependence on $p$ can affect the scaling law of the EWS and hinder the divergence of the observable. Conversely, the operator $D(p)$ is known for simple models \cite{Bernuzzi2024EWSSPDEBoundary,da1993evolution} and may not be dependent on $p$. A related example is described in the section to follow.

The limit $\kappa \to 0^+$ implies the (at most) hyperbolic divergence of the EWS in \eqref{eq:case_1}, \eqref{eq:case_2} and \eqref{eq:case_3}. Such a behaviour is entailed by the structure of  $B_0(p)$ in \eqref{eq:B0_a}, \eqref{eq:B0_b} and \eqref{eq:B0_c}, respectively. While in the limit $p\to 0^-$, at most only one eigenvalue tends to the imaginary axis, in this case, an infinite number of real eigenvalues tend simultaneously to $0^-$ along an equivalent number of eigenfunctions. The assumptions of $Q_1$, or $\Lambda(p)$ in the case of \eqref{eq:case_3}, imply that the corresponding scaling law is captured along a large set of modes in $\mathcal{H}_1$.

The resemblance of the scaling laws in the limit $p\to 0^-$ of the time-asymptotic variance of the models \eqref{eq:case_1}, \eqref{eq:case_2} and \eqref{eq:case_3} compared to the corresponding deterministic models perturbed by white noise \cite{Bernuzzi2024EWSSPDEBoundary,Bernuzzi2024EWSSPDEContinuousSpectrum} indicates that it should be possible to also consider other observables, in analogy to the situation for finite-dimensional dynamics with white noise. Under the assumption of discrete spectrum, a natural example is the time-asymptotic autocorrelation, which is known to behave as an exponential function if studied along eigenmodes. We consider the case of \eqref{eq:case_1}. Then, we construct the time-asymptotic autocovariance with lag time $\tau>0$ as the operator $V_\infty^\tau$ in $\mathcal{H}_1\times \mathcal{H}_1$ such that
\begin{align*}
    \left\langle \begin{pmatrix}
        v_1 \\ v_2
    \end{pmatrix} , V_\infty^\tau \begin{pmatrix}
        w_1 \\ w_2
    \end{pmatrix} \right\rangle_{\mathcal{H}_1\times\mathcal{H}_1}
    = \underset{t\to\infty}{\text{lim}} \text{Cov} \left( \left\langle u(\cdot,t+\tau), v_1 \right\rangle
    + \left\langle \xi_1(\cdot,t+\tau), v_2 \right\rangle, 
    \left\langle u(\cdot,t), w_1 \right\rangle 
    + \left\langle \xi_1(\cdot,t), w_2 \right\rangle \right) ,
\end{align*}
for any $v_1,v_2,w_1,w_2\in\mathcal{H}_1$. Such an operator satisfies \cite[Lemma 3.1]{Bernuzzi2024EWSSPDEBoundary} the equality
\begin{align} \label{eq:lag_lyap}
    V_\infty^\tau = \text{e}^{B_0(p) \tau} V_\infty.
\end{align}
A standard way to employ the time-asymptotic autocorrelation as an early-warning sign is to consider it as the nonlinear operator
\begin{align*}
    \hat{V}_\infty^\tau \left( 
    \begin{pmatrix}
        v_1 \\ v_2
    \end{pmatrix},
    \begin{pmatrix}
        w_1 \\ w_2
    \end{pmatrix}
    \right)
    = \frac{\left\langle \begin{pmatrix}
        v_1 \\ v_2
    \end{pmatrix} , V_\infty^\tau \begin{pmatrix}
        w_1 \\ w_2
    \end{pmatrix} \right\rangle_{\mathcal{H}_1\times\mathcal{H}_1}}{\left\langle \begin{pmatrix}
        v_1 \\ v_2
    \end{pmatrix} , V_\infty \begin{pmatrix}
        w_1 \\ w_2
    \end{pmatrix} \right\rangle_{\mathcal{H}_1\times\mathcal{H}_1}}
\end{align*}
for any $v_1,v_2,w_1,w_2$ such that $\left\langle \begin{pmatrix}
    v_1 \\ v_2
    \end{pmatrix} , V_\infty \begin{pmatrix}
    w_1 \\ w_2
\end{pmatrix} \right\rangle_{\mathcal{H}_1\times\mathcal{H}_1}\neq 0$. From \eqref{eq:B0_a} we define ${\beta_i^{(p)}}^*\in\mathcal{H}_1\times \mathcal{H}_1$ as the eigenfunction of $B_0(p)^*$ corresponding to the eigenvalue $\overline{\lambda_i^{(p)}}$ for any $i\in\mathbb{N}_{>0}$. It can then be proven from \eqref{eq:lag_lyap} that
\begin{align*}
    \hat{V}_\infty^\tau \left( 
    {\beta_i^{(p)}}^*, w
    \right)
    = \text{e}^{\overline{\lambda_i^{(p)}} \tau}
\end{align*}
holds for any $i\in\mathbb{N}_{>0}$ and $w\in\mathcal{H}_1\times \mathcal{H}_1$. A key consideration of this result is the fact that, while
\begin{align*}
    \beta_i^{(p)}:=
    \begin{pmatrix}
        e_i^{(p)} \\ 0
    \end{pmatrix} \in \mathcal{H}_1\times \mathcal{H}_1
\end{align*}
is the eigenfunction of $B_0(p)$ corresponding to $\lambda_i^{(p)}$, the eigenfunctions
\begin{align*}
    {\beta_i^{(p)}}^*=
    \begin{pmatrix}
        {e_i^{(p)}}^* \\ {d_i^{(p)}}^*
    \end{pmatrix} \in \mathcal{H}_1\times \mathcal{H}_1
\end{align*}
of $B_0(p)^*$ are characterized by ${d_i^{(p)}}^*= \frac{\sigma_\text{R}}{\overline{\lambda_i^{(p)}} + \kappa}{e_i^{(p)}}^* \neq 0$. This is an implication of \eqref{eq:B0_a} and, as a consequence, the sole knowledge of the mild solution $u^\text{d}$ of \eqref{eq:case_1} is not sufficient to describe the autocorrelation as an exponential function to be employed as an EWS in the limit $p\to 0^-$. Then, the structure of the stochastic perturbation $\xi_1$ in \eqref{eq:syst_xi} is required to obtain such a construction. In contrast, the time-asymptotic autocorrelation $\hat{V}_\infty^\tau \left( v, w \right) $ for
\begin{align*}
    v=
    \begin{pmatrix}
        v_1 \\ 0
    \end{pmatrix} \in \mathcal{H}_1\times \mathcal{H}_1,
\end{align*}
any $v_1\in\mathcal{H}_1$ and $w\in\mathcal{H}_1\times \mathcal{H}_1$ displays a more complex structure as a function of $\tau$. This structure can be computed explicitly from \eqref{eq:lag_lyap}, depends on $\kappa$, and is deferred for future studies.
Conversely, any
\begin{align*}
    v=
    \begin{pmatrix}
        0 \\ v_2
    \end{pmatrix} \in \mathcal{H}_1\times \mathcal{H}_1
\end{align*}
for $v_2\in\mathcal{H}_1$ is an eigenfunction of $B_0(p)^*$ with eigenvalue $-\kappa$. Hence,
\begin{align*}
    \hat{V}_\infty^\tau \left( 
    v, w
    \right)
    = \text{e}^{-\kappa \tau}
\end{align*}
holds for any $w\in\mathcal{H}_1\times \mathcal{H}_1$ such that $\left\langle v , V_\infty w \right\rangle_{\mathcal{H}_1\times\mathcal{H}_1}\neq 0$. In conclusion, the knowledge of the behaviour of $\xi_1$ in time is sufficient to construct another (false) EWS in the limit $\kappa\to 0^+$ for any $\tau>0$.

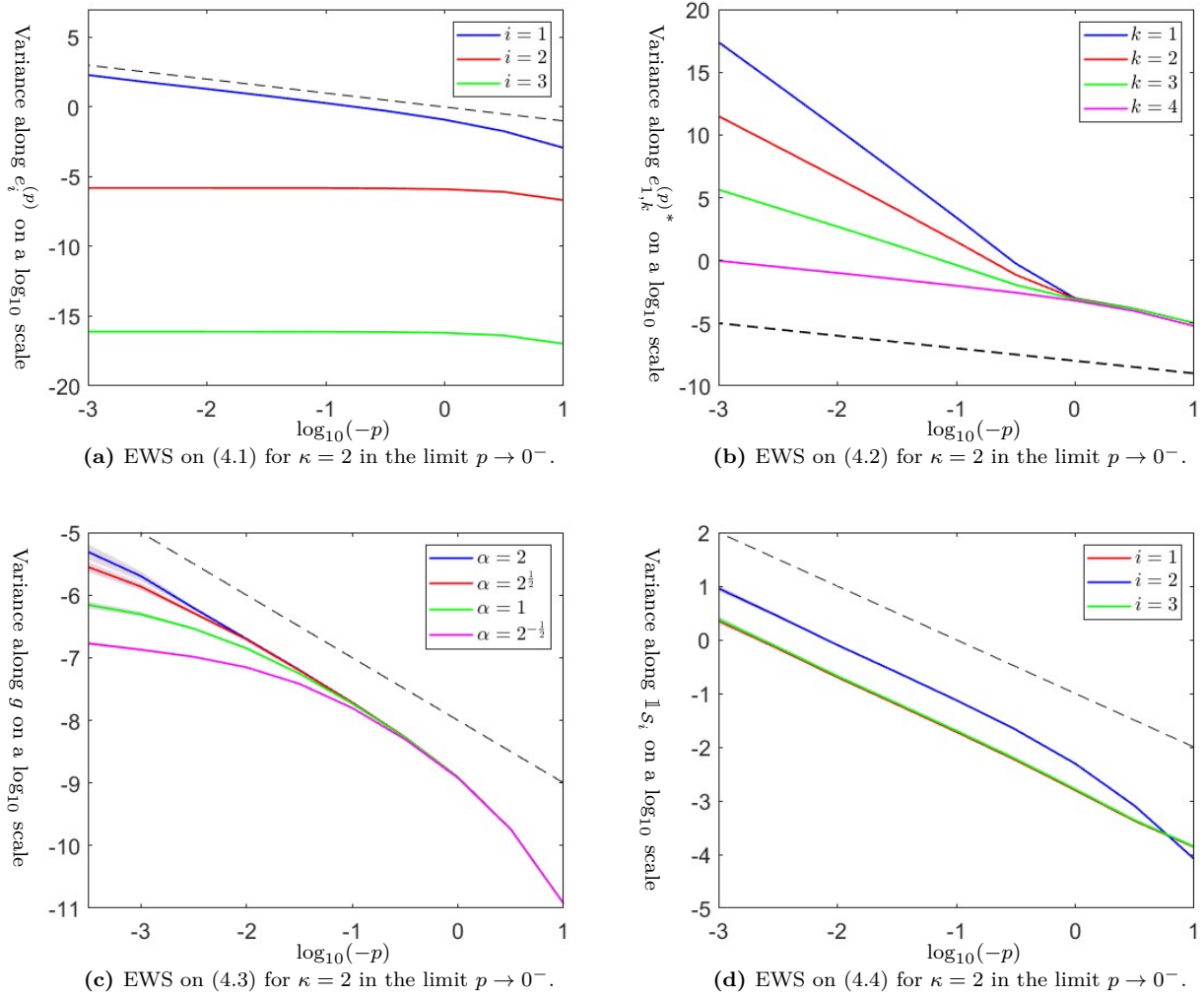
\begin{figure}[h!]
    \centering   
    \subfloat[EWS on \eqref{eq:example_disc} for $\kappa=2$ in the limit $p\to 0^-$.]{\begin{overpic}[width= 0.45\textwidth]{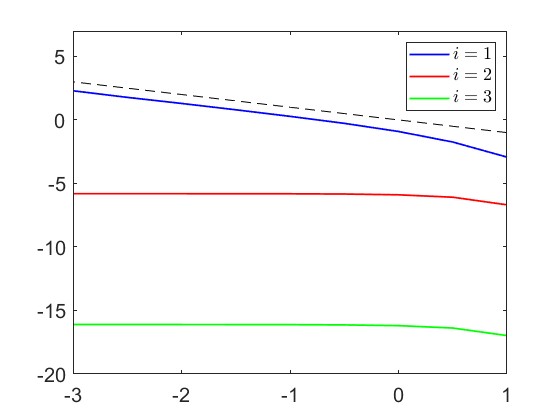}
    \put(470,0){\footnotesize{$\text{log}_{10}(-p)$}}
    \put(0,670){\rotatebox{270}{\parbox{15em}{\footnotesize{Variance along $e_i^{(p)}$ on a $\text{log}_{10}$ scale}}}}
    \end{overpic}
    \label{Fig1_a}}
    \hspace{0mm}
    \subfloat[EWS on \eqref{eq:example_gen} for $\kappa=2$ in the limit $p\to 0^-$.]{\begin{overpic}[width= 0.45\textwidth]{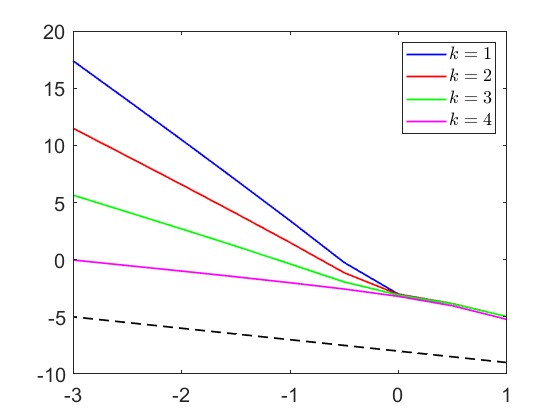}
    \put(470,0){\footnotesize{$\text{log}_{10}(-p)$}}
    \put(0,670){\rotatebox{270}{\parbox{15em}{\footnotesize{Variance along ${e_{1,k}^{(p)}}^*$ on a $\text{log}_{10}$ scale}}}}
    \end{overpic}
    \label{Fig1_b}}

    \vspace{0mm}

    \subfloat[EWS on \eqref{eq:example_cont} for $\kappa=2$ in the limit $p\to 0^-$.]{\begin{overpic}[width= 0.45\textwidth]{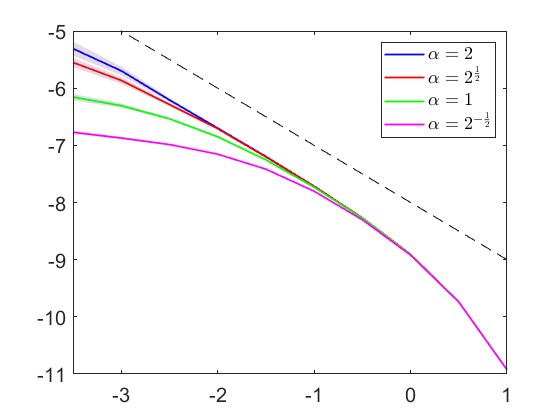}
    \put(470,0){\footnotesize{$\text{log}_{10}(-p)$}}
    \put(0,670){\rotatebox{270}{\parbox{15em}{\footnotesize{Variance along $g$ on a $\text{log}_{10}$ scale}}}}
    \end{overpic}
    \label{Fig1_c}}
    \hspace{0mm}
    \subfloat[EWS on \eqref{eq:example_bound} for $\kappa=2$ in the limit $p\to 0^-$.]{\begin{overpic}[width= 0.45\textwidth]{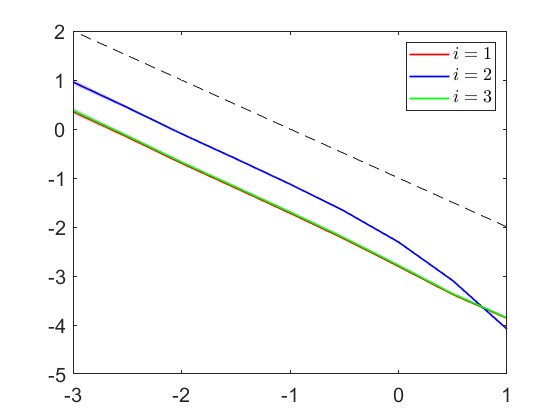}
    \put(470,0){\footnotesize{$\text{log}_{10}(-p)$}}
    \put(0,670){\rotatebox{270}{\parbox{15em}{\footnotesize{Variance along $\mathbbm{1}_{\mathcal{S}_i}$ on a $\text{log}_{10}$ scale}}}}
    \end{overpic}
    \label{Fig1_d}}
    
    \caption{Log-log plots of the variance in time obtained when projecting the SPDE solution along different modes. The limit $p\to 0^-$ is shown from right to left. Each panel corresponds to a different system: $(a)$ the cable equation on an interval with periodic boundary conditions, \eqref{eq:example_disc}; $(b)$ the SDE \eqref{eq:example_gen} with linear drift displaying generalized eigenvectors; $(c)$ the SPDE \eqref{eq:example_cont} with a linear drift term with purely continuous spectrum; $(d)$ the boundary-driven system \eqref{eq:example_bound} with red Dirichlet noise at the extremes of an interval. The values refer to the average of $10$ run samples. The dashed black lines indicate reference hyperbolic scaling laws, whereas the shaded grey regions represent twice the numerical standard deviation. The increase of variance in the projected modes is the manifestation of CSD in these SPDEs, where the drift term approaches a deterministic bifurcation point. Nonetheless, while the scaling law is hyperbolic in $(a)$ and $(d)$, we find examples of enhanced or silenced EWSs in $(b)$ and in $(c)$, respectively.}
    \label{Fig1}
\end{figure}

\section{Numerical Analysis}\label{sec: NumResults}

In this section, we discuss numerical simulations to cross-validate our findings. We numerically solve different types of SPDEs and study the variance of projections along specific modes over a large time interval. Through ergodic theory, such a value approximates the time-asymptotic variance on the corresponding functions in $\mathcal{H}_1$. As such, we substitute the observable with the time variance of the solutions in the time interval $[0,T]$, for $T=10^5$. As discussed in the previous section, we observe different scaling laws in the EWS depending on the assumptions of the systems we consider. Figure \ref{Fig1} and Figure \ref{Fig2} encompass our results by displaying the rate of the observables in log-log plots in the limits $p\to 0^-$ and $\kappa\to 0^+$, respectively. In Figure \ref{Fig1} we fix the value $\kappa=2$, whereas in Figure \ref{Fig2} we consider $p= 0.5$.

In the following examples, the red noise term is the solution of \eqref{eq:syst_xi} for $\sigma=0.1$ and $Q_j=\operatorname{Id}$, the identity operator in $\mathcal{H}_j$ for $j\in\{0,1\}$. The systems are solved through the discretization of the mild solution formula \cite{Bernuzzi2024Large}, unless stated otherwise. The time step is chosen as $\delta t=0.1$ whereas the spatial discretization scale $\delta x$ is fixed in each example. The projections along different modes are computed through the discrete scalar product that approximates the product in $\mathcal{H}_1$. In the corresponding figures, the values are obtained as the average of $10$ independent runs samples, and the initial conditions are set near null functions. The shaded grey areas have a width equal to double the numerical standard deviation in a logarithmic scale to indicate sensibilities in the algorithm. As a reference, we display a dashed black line in each figure that indicates the hyperbolic rate of divergence.

First, we solve the cable equation with periodic boundary conditions. This fundamental reaction-diffusion equation has many modelling applications, e.g., in neuroscience \cite{KOCH1984CableTheoryNeurons, WANG2018ModifiedCableEquation}:
    \begin{align} \label{eq:example_disc}
        \left\{ \begin{alignedat}{2}
            \text{d}u^{\text{d}}(x,t) &= \left( (\Delta + p) u^{\text{d}}(x,t) + \xi_1(x,t) \right) \text{d}t , \\
            u^{\text{d}}(0,t) &= u^{\text{d}}(1,t) ,\\
            u^{\text{d}}(x,0) &= 0 ,
        \end{alignedat} \right.
    \end{align}
for $x\in[0,1]$ and $0<t<T$. The spatial step is $\delta x = 0.005$. We indicate as $\left\{e_i^{(p)}\right\}_{i\in\mathbb{N}_{>0}}$, the eigenfunctions of the selfadjoint differential operator $\Delta+p$ with periodic boundary conditions. We notice then that $\lambda_1^{(p)}=p$. In Figure \ref{Fig1_a}, we display the variance of $\left\langle u^{\text{d}}, f \right\rangle$ for $f=e_i^{(p)}$ and $i\in\{0,1,2\}$ in correspondence to $-p$. Those values are scaled in a logarithmic scale to capture the scaling law in the limit $p\to 0^-$. Of such, only the variance along $f=e_1^{(p)}$ assumes a hyperbolic rate of divergence, while the rest converge in the limit. In contrast, in Figure \ref{Fig2_a}, the scaling law of the variance in the limit $\kappa\to 0^+$ is hyperbolic regardless the eigenmode along which the EWS is observed, since $Q=\operatorname{Id}$. This behaviour is visible through the comparison to the dashed reference line.

In the next, more abstract case, we study an SDE with linear drift component that displays generalized eigenvectors. While this is not a spatial SPDE in the strict sense, it is a useful example to explore the behaviour of different generalized eigenvectors within the same generalized eigenspace. Our example is
    \begin{align}  \label{eq:example_gen}
        \left\{ \begin{alignedat}{2}
            \text{d}u^{\text{g}}(t) &= \left( 
            \begin{pmatrix}
                p & 1 & 0 & 0 \\
                0 & p & 1 & 0 \\
                0 & 0 & p & 1 \\
                0 & 0 & 0 & p
            \end{pmatrix}
            u^{\text{g}}(t) + \xi_1(t) \right) \text{d}t , \\
            u^{\text{g}}(0) &= 0  \in\mathbb{R}^4 ,
        \end{alignedat} \right.
    \end{align}
for $u^{\text{g}}(t), \xi_1(t) \in\mathbb{R}^4 $ and $0<t<T$. In this setting, we consider the discrete spatial space so that $ \mathcal{H}_1=\mathbb{R}^4 $. In Figure \ref{Fig1_b}, we observe the time variance of the solution along the left generalized eigenvectors of the matrix that defines the drift component, or the generalized eigenvectors of its transpose. The only eigenvalue is $\lambda_1^{(p)}=p$ and its corresponding left generalized eigenvectors are
    \begin{align*}
        {e_{1,1}^{(p)}}^*=\begin{pmatrix}
            1 \\ -1 \\ 0 \\ 0
        \end{pmatrix}
        , \quad
        {e_{1,2}^{(p)}}^*=\begin{pmatrix}
            0 \\ 1 \\ -1 \\ 0
        \end{pmatrix}
        , \quad
        {e_{1,3}^{(p)}}^*=\begin{pmatrix}
            0 \\ 0 \\ 1 \\ -1
        \end{pmatrix}
        , \quad
        {e_{1,4}^{(p)}}^*=\begin{pmatrix}
            0 \\ 0 \\ 0 \\ 1
        \end{pmatrix} .
    \end{align*}
As described in Theorem \ref{thm:1}, the scaling laws in the limit $p\to 0^-$ of the time-asymptotic variances are $\Theta_p\left((-p)^{-1}\right)$ along ${e_{1,1}^{(p)}}^*$, $\Theta_p\left((-p)^{-3}\right)$ along ${e_{1,2}^{(p)}}^*$, $\Theta_p\left((-p)^{-5}\right)$ along ${e_{1,3}^{(p)}}^*$ and $\Theta_p\left((-p)^{-7}\right)$ along ${e_{1,4}^{(p)}}^*$. This behaviour is reflected in Figure \ref{Fig1_b} where only the time variance along the first mode is hyperbolic. Conversely, the rate of divergence in the limit $\kappa\to 0^+$ is $\Theta_\kappa\left(\kappa^{-1}\right)$ along each mode, as shown in Figure \ref{Fig2_b}.

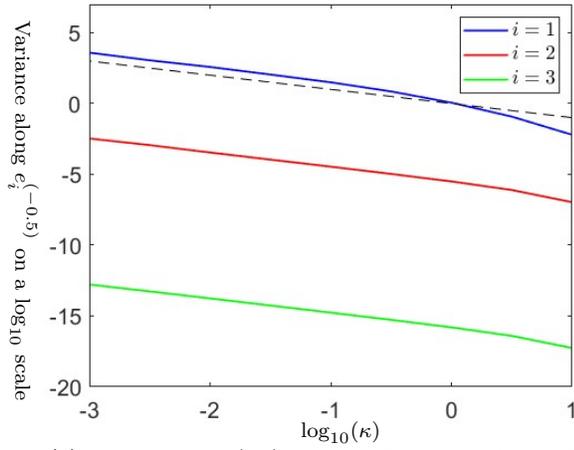
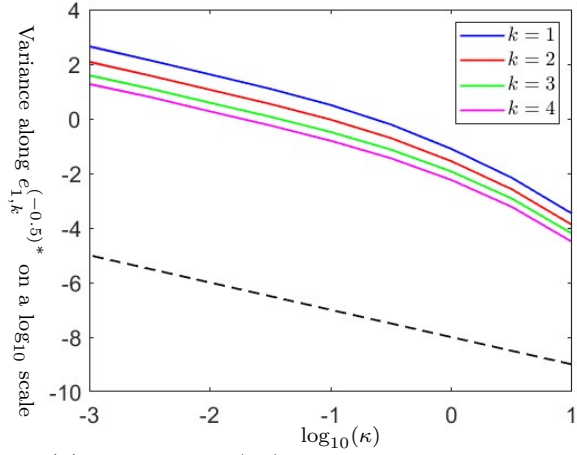
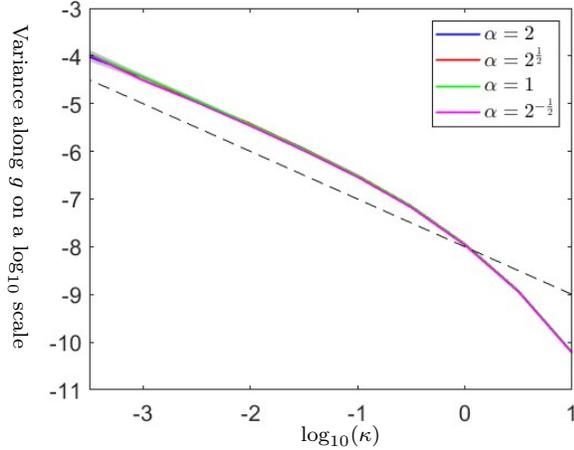
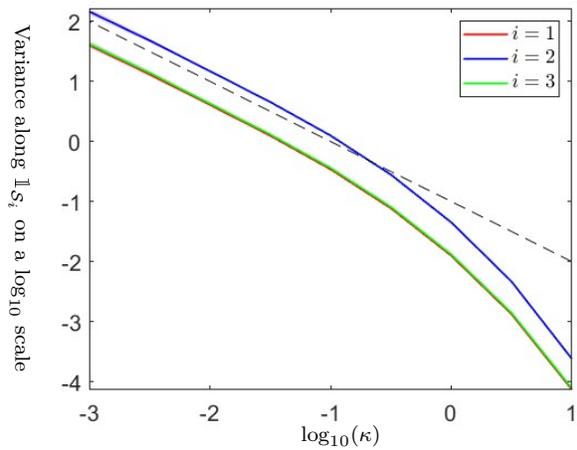
\begin{figure}[h!]
    \centering   
    \subfloat[False EWS on \eqref{eq:example_disc} for $p=0.5$ in the limit $\kappa\to 0^+$.]{\begin{overpic}[width= 0.45\textwidth]{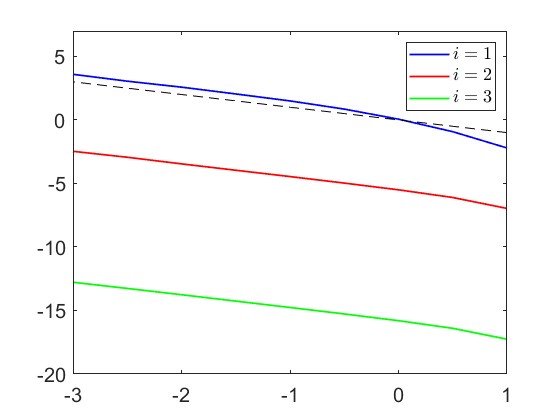}
    \put(470,0){\footnotesize{$\text{log}_{10}(\kappa)$}}
    \put(0,670){\rotatebox{270}{\parbox{15em}{\footnotesize{Variance along $e_i^{(-0.5)}$ on a $\text{log}_{10}$ scale}}}}
    \end{overpic}
    \label{Fig2_a}}
    \hspace{0mm}
    \subfloat[False EWS on \eqref{eq:example_gen} for $p=0.5$ in the limit $\kappa\to 0^+$.]{\begin{overpic}[width= 0.45\textwidth]{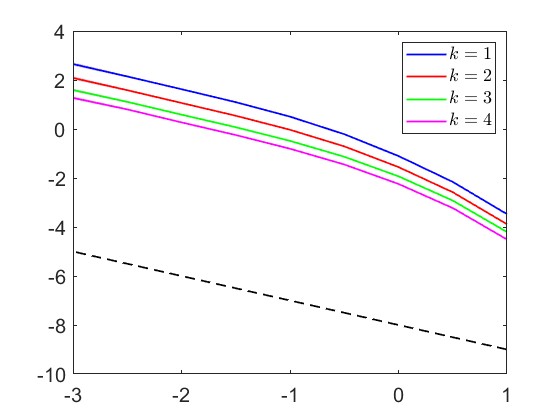}
    \put(470,0){\footnotesize{$\text{log}_{10}(\kappa)$}}
    \put(0,670){\rotatebox{270}{\parbox{15em}{\footnotesize{Variance along ${e_{1,k}^{(-0.5)}}^*$ on a $\text{log}_{10}$ scale}}}}
    \end{overpic}
    \label{Fig2_b}}

    \vspace{0mm}

    \subfloat[False EWS on \eqref{eq:example_cont} for $p=0.5$ in the limit $\kappa\to 0^+$.]{\begin{overpic}[width= 0.45\textwidth]{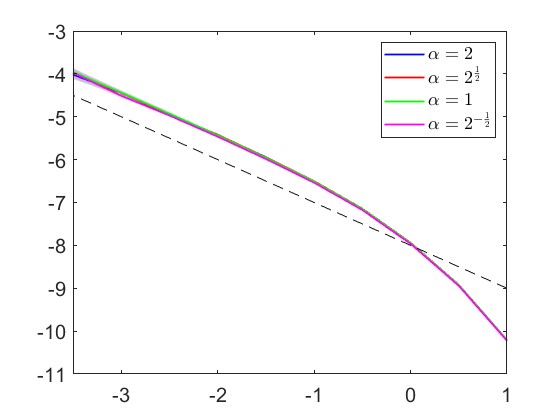}
    \put(470,0){\footnotesize{$\text{log}_{10}(\kappa)$}}
    \put(0,670){\rotatebox{270}{\parbox{15em}{\footnotesize{Variance along $g$ on a $\text{log}_{10}$ scale}}}}
    \end{overpic}
    \label{Fig2_c}}
    \hspace{0mm}
    \subfloat[False EWS on \eqref{eq:example_bound} for $p=0.5$ in the limit $\kappa\to 0^+$.]{\begin{overpic}[width= 0.45\textwidth]{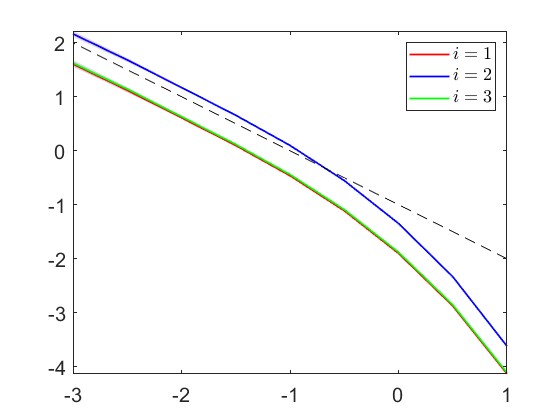}
    \put(470,0){\footnotesize{$\text{log}_{10}(\kappa)$}}
    \put(0,670){\rotatebox{270}{\parbox{15em}{\footnotesize{Variance along $\mathbbm{1}_{\mathcal{S}_i}$ on a $\text{log}_{10}$ scale}}}}
    \end{overpic}
    \label{Fig2_d}}
    
    \caption{Log-log plots of the variance in time obtained when projecting the SPDE solution along different modes. The limit $\kappa\to 0^-$ is shown from right to left. The subfigures $(a)-(d)$ correspond to the same systems as in Figure \ref{Fig1} and their values refer to the average of $10$ run samples. In contrast to the limit $p\to 0^-$, the variance exhibits hyperbolic divergence across all modes due to the noise structure. This is indicated by the alignment of all lines to the dashed black lines, which serve as a reference slope. The grey-shaded regions depict numerical uncertainties. The observed increase in variance, which depends on the increase in the noise correlation time $1/\kappa$, represents a false EWS in the context of CSD.}
    \label{Fig2}
\end{figure}

In Figure \ref{Fig1_c} and Figure \ref{Fig2_c}, we study the time variance of the numerical approximation of the solution associated to
    \begin{align} \label{eq:example_cont}
        \left\{ \begin{alignedat}{2}
            \text{d}u^{\text{c}}(x,t) &= \left( \left(-\lvert x \rvert^\alpha + p \right) u^{\text{c}}(x,t) + \sigma_{\text{R}} \xi_1(x,t) \right) \text{d}t , \\
            \text{d}\xi_1(x,t) &= -\kappa \xi_1(x,t) + \sigma Q_1^\frac{1}{2} \text{d}W_t^1 , \\
            u^{\text{c}}(x,0) &= 0 ,
        \end{alignedat} \right.
    \end{align}
for all $x\in\mathbb{R}$ and $0<t<T$. The spatial grid discretization is set as $\delta x = 10^{-5}$. In this case, we fix the function along which we project $u^{\text{c}}$, which is $g=\mathbbm{1}_{\mathcal{S}}$ for $\mathcal{S}=[-0.01,0.01]$. Instead, we consider different values of $\alpha$ that define the system, as $\alpha\in\left\{2^{-\frac{1}{2}},1,2^{\frac{1}{2}},2\right\}$. As described in Corollary \ref{cor:4}, the time-asymptotic variance of the solution along $g$ displays a rate of divergence less than hyperbolic in the limit $p\to 0^-$. In Figure \ref{Fig1_c}, the lines associated with $\alpha=2^{\frac{1}{2}}$ and $\alpha=2$ approach for small values of $-p$ their expected scaling law, which correspond to the slopes $-1+\frac{1}{\alpha}$ in the log-log scale. For the rest, the slope decreases steadily as $p$ approaches $0$, indicating lower rates of divergence. In Figure \ref{Fig2_c}, the scaling law is equivalent for each value of $\alpha$ and it is hyperbolic in the limit $\kappa\to 0^+$.

Lastly, in Figure \ref{Fig1_d} and Figure \ref{Fig2_d}, we consider the solution of the cable equation with boundary noise
    \begin{align} \label{eq:example_bound}
        \left\{ \begin{alignedat}{2}
            \text{d}u^{\text{b}}(x,t) &= \left( (\Delta + \pi^2 + p) u^{\text{b}}(x,t) \right) \text{d}t , \\
            u^{\text{b}}(0,t) &= \xi_0(0,t) ,\\
            u^{\text{b}}(1,t) &= \xi_0(1,t) ,\\
            u^{\text{b}}(x,0) &= 0 ,
        \end{alignedat} \right.
    \end{align}
for all $x\in[0,1]$ and $0<t<T$. The spatial step is fixed as $\delta x = 0.005$ and the numerical solution is obtained through the implicit Euler method. The space interval $\mathcal{X}_1=[0,1]$ is partitioned into $\mathcal{S}_1=\left[0,\frac{1}{3}\right]$, $\mathcal{S}_2=\left[\frac{1}{3},\frac{2}{3}\right]$ and $\mathcal{S}_3=\left[\frac{2}{3},1\right]$. In the figure, we observe the log-log plot of the time variance of the solution along $f=\mathbbm{1}_{\mathcal{S}_i}$, for $i\in\{1,2,3\}$. For such a system we know the Dirichlet map $D(p):\mathcal{H}_0\to\mathcal{H}_1$ explicitly \cite{da2014stochastic}. It does not depend on $p$ and, as such, satisfies the assumptions in Section \ref{sec: Prel}. This implies the statement of Corollary \ref{cor:6}, which is corroborated by the findings in the figures. In fact, for each $i\in\{1,2,3\}$ we consider $h_1^{(p)}=h_2^{(p)}=\mathbbm{1}_{\mathcal{S}_i}$ for any $p<0$. Then, the time-asymptotic variance along such families of functions displays hyperbolic divergence in the limit $p\to 0^-$, as shown in Figure \ref{Fig1_d}, and in the limit $\kappa\to 0^+$, as displayed in Figure \ref{Fig2_d}.

\section{Discussion and Conclusion}
We have derived expressions for system variance in linear SPDEs under the influence of red noise. The dependence of variance on a critical eigenvalue suggests that in such systems, variance diverges when linear stability is lost (Theorems \ref{thm:1}, \ref{thm:3}, and \ref{thm:5}). This is the case for generic probing functions in the solution space (Corollaries \ref{cor:2}, \ref{cor:4}, and \ref{cor:6}). In this sense, it is reasonable to expect the occurrence of CSD in bifurcating SPDEs with red noise. However, we have also found that a similar divergence takes place when the correlation time of the noise increases. This is problematic, as there is no way to tell the genuine source of an increase in variance in an application setting. The possibility of non-stationary noise characteristics would need to be carefully considered before applying CSD for the detection of approaching bifurcations. We have also discussed a second common EWS for bifurcations, an exponential increase in the autocorrelation. We have shown that such an effect indeed occurs with respect to some specific probing functions. However, also for this EWS, we emphasize the potential for false indications resulting from non-stationary noise.

We have performed numerical experiments for the same class of SPDEs. We introduced red noise either as a dynamic term or as a boundary condition. The expected divergences of variance corresponding to CSD could be reproduced in these experiments. The performed statistical assessment resembles the setting of an applied time series analysis in a real-world system suspected of bifurcating. Furthermore, we have reproduced the effect of an increase in variance as a response to an increase in the correlation time of the noise. The case of muted EWS, on the other hand, could occur when a system genuinely loses linear stability, but noise correlation speeds up simultaneously. The two opposing effects on the system variance can cancel, leading to a muted CSD signal.

In general, the analysis of CSD in time series data is only possible in case the dynamics of the system in question are well-understood \cite{Dakos2015CSDIndicatorsProblems, Kefi2013EWSforNonCatastrophic, Boerlijst2013SilentCollapse, Boettiger2013DiffBifEWS, Ben-Yami2024TippingTime, Morr2024InternalNoiseInterference}. For specific models of real-world systems, it may be possible to derive EWS, such as time series variance or other statistical quantities \cite{Benson2024alphaStableCSD, Morr2024RedNoiseCSD, Morr2024KramersMoyalEWS}. Without confirming that the phenomenon of CSD manifests in the specific modelling setting, its application can potentially fail. In this work, we have added a modelling setting to the list of confirmed occurrences of CSD. This is the bifurcating SPDE setting with red noise. This class of models carries relevance in the analysis of systems from, e.g., the field of climate or ecology \cite{Hasselmann1976Theory1, Zwanzig1961Memory, Vasseur2004ColorEnvironmentalNoise}. The employment of methodology based on CSD is thus mathematically motivated for yet a larger range of real-world systems.

% \section*{Acknowledgments}

\newpage
\bibliographystyle{abbrv}
\bibliography{Ref}

\newpage
\appendix
    
\end{document}